\documentclass[preprint,12pt]{elsarticle}

\usepackage{amssymb}
\usepackage{amsmath}
\usepackage{amsthm}

\journal{J. Math. Pures Appl.}

\usepackage{mathtools}
\usepackage{stmaryrd}
\usepackage{enumitem}
\usepackage{todonotes}

\newtheorem{theorem}{Theorem}
\newtheorem{lemma}[theorem]{Lemma}
\newtheorem{proposition}[theorem]{Proposition}
\newdefinition{rmk}{Remark}

\newcommand{\bel}{\begin{equation} \label}
\newcommand{\ee}{\end{equation}}

\def\p{\partial}
\def\R{\mathbb R}
\def\o2{\overline{O_2}}

\def\V{\mathcal V}

\def\dist{\text{dist}}

\DeclareMathOperator{\supp}{supp}
\DeclareMathOperator{\Op}{Op}
\DeclareMathOperator{\scl}{scl}
\DeclareMathOperator{\Hess}{Hess}
\DeclareMathOperator{\Span}{span}

\newcommand{\proj}{\mathcal{P}} 
\newcommand{\compl}{\mathcal{Q}}
\DeclarePairedDelimiter{\norm}{\lVert}{\rVert}
\newcommand{\tnorm}[1]{\vert\hspace{-0.3mm}\Vert#1\Vert\hspace{-0.3mm}\vert}
\newcommand{\tangular}[1]{ \llbracket\kern-0.5ex|#1|\kern-0.5ex\rrbracket} 
\newcommand{\jump}[1]{\llbracket#1\rrbracket}

\providecommand{\abs}[1]{\left\lvert#1\right\rvert}

\newcommand{\dT}{\mathrm{d}t}
\newcommand{\dX}{\mathrm{d}x}
\newcommand{\dS}{\mathrm{d}S}
\newcommand{\STdom}{Q}
\newcommand{\STdata}{\omega_T}

\newcommand{\Qmin}{\tilde{Q}}
\newcommand{\FiniteDimSpace}{\mathcal{V}}

\newcommand{\FullyDiscrSpace}[2]{ W^{ {#1},{#2}}_{h } }

\newcommand{\ProdFullyDiscrSpace}[2]{ \mathcal{W}^{ {#1},{#2}}_{h } }
\providecommand{\StabCIP}[2]{ J \left( #1, #2 \right) }
\providecommand{\StabGLS}[2]{ G  \left( #1, #2 \right) }
\providecommand{\StabTikh}[2]{ J_{ \gamma }  \left( #1, #2 \right) }
\providecommand{\StabDt}[2]{ I_{ 0 }  \left( #1, #2 \right) }
\providecommand{\StabTrace}[2]{ R_{ \mathcal{Q} } \left( #1, #2 \right) }
\providecommand{\StabTimeJumps}[2]{ S^{\uparrow \downarrow}_{h}  \left( #1, #2 \right) }

\begin{document}

\begin{frontmatter}


\title{Unique continuation for the wave equation: the stability landscape\tnoteref{label1}}
\tnotetext[label1]{Funding by EPSRC grant EP/V050400/1 is gratefully acknowledged.
L.O. was supported by the European Research Council of the European Union, grant 101086697 (LoCal), and the Reseach Council of Finland, grants 347715, 353096 and 359182. Z.Z was supported by the Finnish Ministry of Education and Culture’s Pilot for Doctoral Programmes (Pilot project Mathematics of Sensing, Imaging and Modelling). Views and opinions expressed are those of the authors only and do not necessarily reflect those of the European Union or the other funding organizations.
}
\author[labela]{Erik Burman}

\ead{e.burman@ucl.ac.uk}
	\affiliation[labela]{organization={Department of Mathematics, University College London},
	     city={London},
	     country={UK}}



\author[labelb]{Lauri Oksanen} 
\ead{lauri.oksanen@helsinki.fi}
	\affiliation[labelb]{organization={Department of Mathematics and Statistics, University of Helsinki},
            city={Helsinki},
            country={Finland}}

\author[labelc]{Janosch Preuss} 
\ead{janosch.preuss@inria.fr}
	\affiliation[labelc]{organization={Project-Team Makutu, Inria Bordeaux, University of Pau and Pays de l’Adour},
            city={Pau},
            country={France}}

\author[labelb]{Ziyao Zhao} 
\ead{ziyao.zhao@helsinki.fi}

\begin{abstract}
We consider a unique continuation problem for the wave equation given data in a volumetric subset of the space time domain. In the absence of data on the lateral boundary of the space-time cylinder we prove that the solution can be continued with H\"older stability into a certain proper subset of the space-time domain. Additionally, we show that unique continuation of the solution to the entire space-time cylinder with Lipschitz stability is possible given the knowledge of a suitable finite dimensional space in which the trace of the solution on the lateral boundary is contained. These results allow us to design a finite element method that provably converges to the exact solution at a rate that mirrors the stability properties of the continuous problem.
\end{abstract}




\end{frontmatter}



\section{Introduction}\label{section:intro}
Wave propagation problems with unknown boundary or initial data arise in a wide range of imaging applications, from medical diagnostics to seismological exploration. Such problems are computationally extremely challenging, due to the lack of stability of the reconstruction problem and the strong space-time coupling. In addition the complex stability properties of the wave equation makes it difficult to characterise exactly what accuracy a computation can be expected to have.  It is known that there is a strong interplay between the geometry of the data, and in the target domains as well as the  a priori assumptions made on the exact solution. It is therefore of interest to establish a theory for the approximation of ill-posed problems where the approximation can be shown to produce the best possible result locally under the constraint of the physical stability. The objective of the present paper is to take some steps in this direction. In particular we establish H\"older stability of the solution under a special convexity condition on the data set and global Lipschitz stability if the trace of the solution is in (or close to) some finite dimensional set. We then design a numerical method which reproduces the accuracy indicated by these stability estimates. This paper is organized as follows: in the remainder of this section we introduce the considered problem and give an overview of the results. In Section \ref{section:stab} we prove the different stability estimates. The finite element method is introduced in Section \ref{section:FEM} and an error analysis based on the stability estimates is proposed in Section \ref{section:error-analysis}. Finally the practical ramifications of the theoretical results are investigated in a series of numerical experiments in Section \ref{section:numexp}.
\subsection{Statement of the considered problem}
This article deals with an ill-posed unique continuation problem for the wave equation which is defined as follows. Let $\Omega \subset \mathbb{R}^d$, where $d \in \mathbb{N}$ be a bounded domain with smooth strictly convex boundary $\partial \Omega$. Given an interval $I_T = (-T,T)$ or $I_T = (0,T)$ with $T > 0$ we denote the space-time cylinder $\STdom := I_T \times \Omega$ and its lateral boundary as $\Sigma := I_T \times \partial \Omega$.
We consider a data assimilation problem with $L^2$-data $u_{\omega}$ given in $ \STdata := I_T \times \omega$ for a subset $\omega \subset \Omega$ which is subject to the wave equation in $ \STdom$, i.e.\
\bel{eq:PDE+data_constraint} 
 \Box u  = 0 \text{ in } \STdom, \quad  u  = u_{\omega} \text{ in }  \STdata, 
\ee
where $\Box := \p^2_t-\Delta$ is the linear wave operator. 
The objective is then to extend $u$ from the data set $\STdata$ to a subset of  $\STdom$. To give more precise statements, stability estimates are required. \par
From \cite[Remark A.5]{BFMO} the following stability estimates is known 
\bel{eq:Lipschitz-stability}
\norm{u}_{L^{\infty}([0,T];L^2(\Omega))} + \norm{\partial_t u}_{L^{2}([0,T];H^{-1}(\Omega))} 
	\lesssim  \norm{u}_{ L^2(\STdata) } + \norm{ \Box u}_{ H^{-1} ( \STdom ) } + \norm{u}_{ L^2(\Sigma)}.
\ee
Note that $\norm{u}_{ L^2(\Sigma)}$ appears on the right hand side. Hence, under the condition that additional data on the lateral boundary $\Sigma$ is available the problem enjoys Lipschitz stability. For this case a zoo of numerical methods, see e.g.\ \cite{CM15,BFO20,BFMO21,MM21,DMS23,BP24_wave,BDE25}, have been proposed. Most of them lead to error estimates that are optimal with respect to the stability of the continuous problem. We refer to \cite[Section 1.1]{BP24_wave} for a detailed discussion of the literature on this and related problems. Our stability analysis is also of interest in the context of control for wave equations. There is a rich literature on this topic, and we refer the reader to \cite{RUss71,Lions88,Robb95,FI96,BG97,MIll02,RLTT17}. Recently the geometric properties of the observability was explored in a context different to ours in \cite{DEZ25} leading to stability estimates that could be of interest for numerical approximation of the type we propose herein.\par
The case in which no data on $\Sigma$ is available, which is the subject of this article, appears to be far less understood. Our aim in this contribution is first to characterize the stability 
of this problem through the derivation of conditional stability estimates. These estimates will be given in such a form that they can subsequently be employed to drive the design of a finite element method that can be used to compute an approximate solution to problem \eqref{eq:PDE+data_constraint}.
This approximation can be shown to converge to the exact solution at a rate which reflects the inherent stability properties of the continuous problem.

\subsection{Overview of results}
We give a brief overview of the new results established in this article without going into all the technical details. In Section \ref{section:stab} we derive the following results characterizing the stability of problem \eqref{eq:PDE+data_constraint}.  
\begin{itemize}
\item If no information on the trace of $u$ on $\Sigma$ is available, then a H\"older stability result in a certain proper subset $B$ of the space-time domain holds. That is, there exist $\alpha \in (0,1)$ such that 
\begin{equation}\label{eq:Hoelder-stab-intro}
	\norm{u}_{L^2(B)} 
\lesssim
(\norm{u}_{L^2( \STdata  )} 
+ \norm{\Box u}_{H^{-1}(\Qmin)})^\alpha 
\norm{u}_{L^2(\Qmin)}^{1-\alpha}, \quad u \in H^1(\Qmin),
\end{equation}
for a certain domain $\Qmin \subset \STdom$. To derive this estimate we restrict ourselves 
to a particular geometry $\Omega$ and choose a symmetric time interval $I_T = (-T,T)$ around the origin. 
\item Given a finite dimensional subspace $\FiniteDimSpace$ of $L^2(\Sigma)$ in which the trace of the solution on $\Sigma$ is contained allows to recover Lipschitz stability in the entire space-time cylinder $\STdom$. To derive this result we require in addition to the geometric control condition \cite{BFMO, BLR88} that the support of the functions in the finite dimensional space is sufficiently broad, see assumption \eqref{eq:assum_2} for a precise statement. Then we obtain 
\begin{align}
& \norm{u}_{L^{\infty}([0,T];L^2(\Omega))} + \norm{\p_t u}_{L^{2}([0,T];H^{-1}(\Omega))} \nonumber \\[3mm] 
& \lesssim \norm{u}_{L^2(\STdata)} + \norm{\square u}_{H^{-1}(\STdom)}  
+ \norm{\compl u}_{L^2(\Sigma)}, \quad u \in  H^1(\STdom), \label{eq:Lipschitz-stab-intro}
\end{align}
 where $\compl = 1 - \proj$ with $\proj$ being the projection operator 
onto $\FiniteDimSpace$.  
\end{itemize}
In Section \ref{section:FEM} we then propose a stabilized finite element method that can be used to compute a discrete approximation $u_h$ to the true solution. Making use of the results from Section \ref{section:stab} we prove in Theorem \ref{thm:error-estimate-fully-discrete} the following convergence rates: 
\begin{itemize}
\item  Given access to noisy data $u_{\omega} + \delta$ on $\STdata$ we obtain in the setting of the H\"older stability result \eqref{eq:Hoelder-stab-intro} that 
\[
\norm{ u - u_h }_{L^2(B)} \lesssim  h^{\alpha s} \left( \norm{u}_{ H^{s+2}(\STdom)  } + h^{-s} \norm{ \delta u }_{L^2(\STdata) }  \right),
\]
where $s$ denotes the minimal polynomial degree of the shape functions in space and time and $h$ is the mesh width.
\item Given a finite dimensional subspace $\FiniteDimSpace$ satisfying \eqref{eq:assum_2} such that $u|_{\Sigma} \in \FiniteDimSpace$ we obtain
\begin{align*}
\norm{ u - u_h }_{L^{\infty}([0,T];L^2(\Omega))} & + \norm{\p_t (u - u_h) }_{L^{2}([0,T];H^{-1}(\Omega))} \\[3mm]
 & \lesssim h^{s}  \norm{u}_{ H^{s+2}(\STdom)} + \norm{ \delta u }_{L^2(\STdata) }. 
\end{align*}
\end{itemize}
These convergence rates clearly mirror the H\"older, respectively Lipschitz stability property of the continuous problem. In Section \ref{section:numexp} we present numerical experiments to validate our theoretical results and explore further aspects not covered by our theory. For example, in the setting of estimate \eqref{eq:Hoelder-stab-intro} we investigate whether the numerical solution converges also in the complement $\STdom \setminus B$. We also check the consequences of fulfilling the condition 
$u|_{\Sigma} \in \FiniteDimSpace$ only approximately which is of practical relevance in applications. 

\section{Stability results}\label{section:stab}

\subsection{Conditional H\"older stability}\label{ssection:stab-hoelder}

The Carleman estimate by H\"ormander \cite[Th. 28.2.3]{hormander1985} forms a basis of unique continuation results with conditional H\"older stability. A classical version of such a result is \cite[Th. 3.4.1]{isakov1998}. We refer to \cite{hoop2018} and \cite{stefanov2019} for geometric reformulations of this result, with applications in inverse problems. 

For simplicity, we consider here a variant with data given on a subset of the bulk rather than of the boundary. Moreover, we restrict our attention to a very concrete geometry that is a spacetime analogue of the geometry in Figure~1 of \cite{burman2019}. Similar geometry is also considered in Figure~1 of \cite{stefanov2019}. Our primary purpose is to shift the Carleman estimate in the Sobolev scale so that the resulting unique continuation estimate is suitable for the subsequent error analysis of the finite element method.  

We write $(t, x) = (t, x^1, \dots, x^d) \in \R^{1+d}$ for the time and space variables, and
    \begin{align}
\Box = \p_t^2 - \Delta.
    \end{align}
We let $0 < r < R$ and define
    \begin{align}
H = \{(x^1, \dots, x^d) \in \R^d \mid x^1 < 0 \},
\quad
\Omega = H \cap B(0, R), 
\quad 
\omega = \Omega \setminus \overline{B(0, r)}.
    \end{align}
Let $\beta > 0$, $\epsilon > 0$ and define
    \begin{align}
\psi(t,x) = |x - y|^2 - (1-\epsilon) t^2
    \end{align}
where $y = (\beta, 0, \cdots, 0) \in \R^d$.
Define, further,
    \begin{align}
\rho_0 = r^2 + \beta^2, 
\quad 
\rho_1 = (r + \beta)^2,
    \end{align}
and let 
    \begin{align}
\rho \in (\rho_0, \rho_1), 
\quad
T > \sqrt{\frac{\rho_1 - \rho}{1-\epsilon}}.
    \end{align}
We write
 \begin{align}
\mho(s) = \{ (t,x) \in (-T,T) \times \Omega \mid \psi(t,x) > s\} \label{eq:mho-def}
 \end{align}
and define 
\begin{equation}\label{eq:B-def-hoelder}
B = \mho(\rho).
\end{equation}
This set is visualized in Figure \ref{fig_geom}.
Finally, we let $\delta \in (0, \rho)$ and define $\Qmin = \mho(\delta)$.

\begin{figure}
\centering
\includegraphics[width=0.5\textwidth,trim={2.5cm 4cm 4cm 2cm},clip]{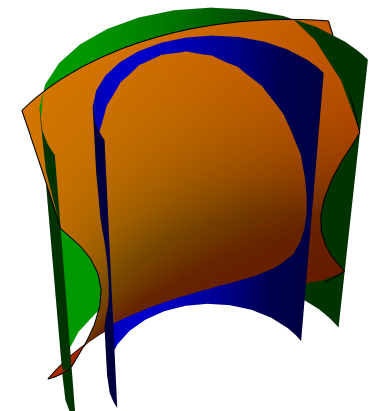}
\caption{The spacetime geometry. Time is the vertical axis. The curved parts of the boundaries of $(-T,T) \times \Omega$ and $(-T,T) \times \omega$ in green and blue, respectively.
The surface $\psi = \rho$ in orange. 
Set $B \setminus \omega$ is the region between the blue and orange surfaces. 
}
\label{fig_geom}
\end{figure}

\begin{lemma}\label{lem:stab}
There is $\alpha \in (0,1)$ such that 
\begin{align}\label{eq:Hoelder-stab}
\norm{u}_{L^2(B)} 
\lesssim
(\norm{u}_{L^2( \STdata  )} 
+ \norm{\Box u}_{H^{-1}(\Qmin)})^\alpha 
\norm{u}_{L^2(\Qmin)}^{1-\alpha}, \quad u \in H^1(\Qmin).
    \end{align}
\end{lemma}
\begin{proof}
The spacetime Hessian of $\psi$ has the simple form
    \begin{align}
\Hess \psi = 2 \begin{pmatrix}
-(1-\epsilon)
\\ & 1
\\ && \ddots
\\ &&& 1
\end{pmatrix}.
    \end{align}
In particular, if $X = (a, \theta) \in \R^{1+d}$ is a null vector, that is, if $|\theta| = |a|$, then 
    \begin{align}
\label{psi_conv}
\Hess \psi(X, X) = a^2(-(1-\epsilon) + 1) = a^2\epsilon > 0,
    \end{align} 
assuming that $a \ne 0$.
Moreover, in $Q$ there holds
    \begin{align}
- |\p_t \psi|^2 + |\nabla \psi|^2  = 2 \psi > 0.
    \end{align}
In particular, the spacetime differential $d\phi$ of $\psi$ is not null.
This and \eqref{psi_conv} imply that the level surfaces of $\psi$ are pseudo-convex in $\Qmin$. Setting $\phi = e^{\lambda \psi}$ for large enough $\lambda > 0$, it follows from \cite[Th. 28.2.3, Prop. 28.3.3]{hormander1985} that 
    \begin{align}
\label{carleman}
\int_{\Qmin} (\tau^3 |w|^2 + \tau |d w|^2) e^{2 \tau \phi} dx 
\lesssim
\int_{\Qmin} |\Box w|^2 e^{2 \tau \phi} dx, 
\quad
w \in C^\infty_0(\Qmin),\ \tau > 1.
    \end{align}

Let $s_1, s_2 \in (\max(\delta, \rho_0), \rho)$ satisfy 
$s_1 < s_2$ and choose 
$\chi_j \in C_0^\infty(\Qmin)$, $j=1,2$, such that, recalling the notation $\omega_T = (-T,T) \times \omega$, there holds
\begin{itemize}
\item[(i)] $\chi_1 = 1$ in $\mho(s_1) \setminus \omega_T$,
\item[(ii)] $\supp(\chi_2) \subset (\Qmin \setminus \mho(s_2)) \cup \omega_T$,
\item[(iii)] $\chi_2 = 1$ in $\supp(d\chi_1)$.
\end{itemize}
Moreover, choose $\chi_3, \chi_4 \in C_0^\infty(\Qmin)$ 
such that $\chi_{4} = 1$ in $\supp \chi_{3}$ 
and that $\chi_{3} = 1$ in $\supp \chi_{1}$.

We view $\hbar = \tau^{-1}$ as a semiclassical parameter. Then the conjugated and rescaled operator $P = e^{\phi/\hbar} \circ \hbar^2 \Box \circ e^{-\phi/\hbar}$ is a semiclassical differential operator. The Carleman estimate \eqref{carleman} can be rewritten as 
    \begin{align}
\label{carleman_scl}
\hbar \norm{\tilde w}_{H^1_{\scl}(\R^{1+d})}^2 \lesssim 
\norm{P \tilde w}_{L^2(\R^{1+d})}^2,
\quad \tilde w \in C^\infty_0(\Qmin),\ \hbar \in (0, 1),
    \end{align}
where the semiclassical norm is defined for any $s \in \R$ by
    \begin{align}
\norm{u}_{H^s_{\scl}(\R^{1+d})} = \norm{J^s u}_{L^2(\R^{1+d})},
\quad u \in H^s(\R^{1+d}),
    \end{align}
see (8.3.5) of \cite{zworski2012}.
Here
    \begin{align}
\quad J^s = \Op_1(j^s), \quad j^s(\xi) = (1 - |\xi|^2)^{s/2},
\quad \xi \in \R^{1+d},
    \end{align}
and $\Op_1$ is the standard quantization in the sense of Section~4.1 of \cite{zworski2012}. In other words, $J^s$ is the semiclassical pseudodifferential operator corresponding to the symbol $j^s$.

For $v \in H^1(\R^{1+d})$ there holds
    \begin{align}
\norm{\chi_1 v}_{L^2(\R^{1+d})} 
= 
\norm{\chi_4 J^{-1} (\chi_1 v)}_{H^1_{\scl}(\R^{1+d})}
+
\norm{(1-\chi_4) J^{-1} (\chi_3 \chi_1 v)}_{H^1_{\scl}(\R^{1+d})},
    \end{align}
and the pseudolocality
    \begin{align}
\norm{(1-\chi_4) J^{-1} (\chi_3 u)}_{H^1_{\scl}(\R^{1+d})}
\lesssim \hbar \norm{u}_{L^2(\R^{1+d})}, 
\quad u \in L^2(\R^{1+d}),
    \end{align}
see e.g. (4.8) of 
\cite{dos-santos-ferreira2009}, implies for small enough $\hbar > 0$
    \begin{align}
\norm{\chi_1 v}_{L^2(\R^{1+d})} 
\lesssim 
\norm{\chi_4 J^{-1} (\chi_1 v)}_{H^1_{\scl}(\R^{1+d})}.
    \end{align}
We apply \eqref{carleman_scl} to $\tilde w = \chi_4 J^{-1} (\chi_1 v)$ and obtain 
    \begin{align}
\sqrt{\hbar} \norm{\chi_1 v}_{L^2(\R^{1+d})} 
\lesssim 
\norm{P (\chi_4 J^{-1} (\chi_1 v))}_{L^2(\R^{1+d})}.
    \end{align}
The commutator estimate
    \begin{align}
\norm{[P,\chi_4 J^{-1}] u}_{L^2(\R^{1+d})} 
\lesssim
\hbar \norm{u}_{L^2(\R^{1+d})}, 
\quad
u \in L^2(\R^{1+d}),
    \end{align}
see e.g. (12) of \cite{burman2019} implies for small enough $\hbar > 0$
    \begin{align}
\sqrt{\hbar} \norm{\chi_1 v}_{L^2(\R^{1+d})} 
\lesssim 
\norm{\chi_4 J^{-1} P (\chi_1 v))}_{L^2(\R^{1+d})}
\lesssim 
\norm{P (\chi_1 v)}_{H^{-1}_{\scl}(\R^{1+d})}.
    \end{align}
Using (iii) we have 
    \begin{align}
\sqrt{\hbar} \norm{\chi_1 v}_{L^2(\R^{1+d})} 
\lesssim 
\norm{\chi_1 P v}_{H^{-1}_{\scl}(\R^{1+d})}
+
\norm{[P, \chi_1] (\chi_2 v)}_{H^{-1}_{\scl}(\R^{1+d})},
    \end{align}
and the commutator estimate
    \begin{align}
\norm{[P, \chi_1] u}_{H^{-1}_{\scl}(\R^{1+d})} 
\lesssim
\hbar \norm{u}_{L^2(\R^{1+d})}, 
\quad
u \in L^2(\R^{1+d}),
    \end{align}
gives
    \begin{align}
\sqrt{\hbar} \norm{\chi_1 v}_{L^2(\R^{1+d})} 
\lesssim 
\norm{\chi_1 P v}_{H^{-1}_{\scl}(\R^{1+d})}
+
\hbar \norm{\chi_2 v}_{L^2(\R^{1+d})}.
    \end{align}
Letting $u \in H^1(\R^{1+d})$ and taking $v = e^{\phi / \hbar} u$,
we have
    \begin{align}
\sqrt{\hbar} \norm{e^{\phi / \hbar}\chi_1 u}_{L^2(\R^{1+d})} 
\lesssim 
\norm{e^{\phi / \hbar} \chi_1 \hbar^2 \Box u}_{H^{-1}_{\scl}(\R^{1+d})}
+
\hbar \norm{e^{\phi / \hbar} \chi_2 u}_{L^2(\R^{1+d})}.
    \end{align}

We pass back to the notation $\tau = \hbar^{-1}$
and have for large $c > 0$ that 
    \begin{align}
\norm{e^{\phi / \hbar} \chi_1 \hbar^2 \Box u}_{H^{-1}_{\scl}(\R^{1+d})}
\le e^{c\tau} \norm{\Box u}_{H^{-1}(\Qmin)}.
    \end{align}
Moreover, using (ii)
    \begin{align}
\hbar \norm{e^{\phi / \hbar} \chi_2 u}_{L^2(\R^{1+d})}
\le 
e^{\tau e^{\lambda s_2}} \norm{u}_{L^2(\Qmin)}.
    \end{align}
Recall that $\phi > \rho$ on $B$. Hence it follows from (i)
that 
    \begin{align}
e^{\tau e^{\lambda \rho}} \norm{u}_{L^2(B \setminus \omega_T)}
\lesssim 
\sqrt{\hbar} \norm{e^{\phi / \hbar}\chi_1 u}_{L^2(\R^{1+d})}.
    \end{align}
Therefore, writing
$\kappa = e^{\lambda \rho} - e^{\lambda s_2} > 0$, we have
    \begin{align}
\norm{u}_{L^2(B)}^2
\lesssim
e^{2 c \tau}(\norm{\Box u}_{H^{-1}(\Qmin)}^2 
+ \norm{u}_{H^1(\omega_T)}^2) 
+ e^{-2 \tau \kappa} \norm{u}_{L^2(\Qmin)}^2.
    \end{align}
The claim follows from \cite[Lem. 5.2]{le-rousseau2012}.
\end{proof}

\subsection{Finite dimensional trace}
In this section, we consider the stability for the unique continuation of the wave equation with a priori knowledge that the trace lies in a finite dimensional space. 
 The equation is formulated as
  \begin{equation}
    \label{eq:problem}
    \begin{cases}
      \square u = f, \text{ in }[0,T]\times \Omega,\\
      u|_{[0,T]\times\p\Omega}\in \FiniteDimSpace,\\
      u|_{[0,T]\times \omega}=q,
    \end{cases}
  \end{equation}
  where $\FiniteDimSpace \subset L^2([0,T]\times \p\Omega)$, is a finite dimensional spaces. We denote the projection operator on $\FiniteDimSpace$ by $\proj$, and the corresponding complement operator is written as $\compl$.\par
  We will assume that $(0,T)\times\omega$ satisfies the geometric control condition in $(0,T)\times \Omega$. The condition is, roughly speaking, every line with lightlike tangent vector, reflecting at $(0,T)\times \p \Omega$ according to Snell's law, intersects $(0,T)\times \omega$. (A vector $v=(v^0,v^1)\in \mathbb{R}^{1+n}$ is lightlike if $\abs{v^0}=\abs{v^1}$.) We refer to \cite{BLR} for the precise definition.\par
  Next we introduce the set
  \begin{equation}
    \label{eq:def_gamma}
    \Gamma:=\left\{(t,x)\in [0,T]\times \p\Omega\mid \dist(x,\omega)\leq \frac{T}{2}-\abs{t-\frac{T}{2}}\right\}.
  \end{equation}
  We claim that the geometric control condition implies that $\Gamma^\circ$ is not empty. To show this, we firstly verify that $\dist(\p\Omega,\overline{\omega})<\frac{T}{2}$. For simplicity, we denote $\dist(\p\Omega,\overline{\omega})$ by $s$ and we assume that $s>0$ without loss of generality. Since $\p\Omega$ and $\overline{\omega}$ are compact, we can find $x_0\in \p\Omega$ and $y_0\in \p\omega$ such that
  \begin{equation}
    \label{eq:minimal_distance_line}
    d(x_0,y_0)=\min_{x\in \p\Omega,\ y\in \overline{\omega}}d(x,y)=s.
  \end{equation}
  Let $\gamma(t)$ be a unit speed straight line that joins $y_0$ and $x_0$ with $\gamma(0)=y_0$ and $\gamma(s)=x_0$. Notice that $\gamma((0,s))$ lies in $\Omega\setminus \overline{\omega}$ since $\gamma(t)\not\in \p\Omega\cup\overline{\omega}$ for all $0<t<s$, otherwise it contradicts \eqref{eq:minimal_distance_line}. Since $\gamma$ realizes the minimal distance from $y_0$ to $\p\Omega$, $\gamma^\prime(s)$ is normal to $\p\Omega$ at $x_0$. Then the reflection of $\gamma$ at the boundary is the curve
  \begin{equation}
    \widetilde{\gamma}(t):=\begin{cases}
      \gamma(t),\ 0<t\leq s,\\
      \gamma(s-t),\ s<t\leq 2s.
    \end{cases}
  \end{equation}
    Since $\widetilde{\gamma}([0,2s])$ does not intersect with $\omega$, the geometric control condition reads that $T>2s$. Choose $0<\varepsilon<\frac{1}{6}(T-2s)$. Hence we have for any $y\in B(y_0,\varepsilon)\cap \omega$, $x\in B(x_0,\varepsilon)\cap \p\Omega$ and $\frac{T}{2}-\varepsilon<t<\frac{T}{2}+\varepsilon$, there holds
  \begin{equation}
    d(x,y)\leq d(x,y_0)+d(y_0,x_0)+ d(x_0,x)=s+2\varepsilon=\frac{T}{2}-\varepsilon\leq \frac{T}{2}-\abs{t-\frac{T}{2}}.
  \end{equation}
  This proves our claim.
   
    We assume that
  \begin{align}
    \tag{A1}
    \label{eq:assum_2}
    \forall \varphi\in \FiniteDimSpace,\ \text{if }\varphi|_\Gamma=0,\ \text{then }\varphi=0.
  \end{align}
  Then we have the following Lipschitz stability.
  \begin{theorem}\label{thm:lipschitz-stab-fin}
    Let $\FiniteDimSpace$ be a finite dimensional subspace of $L^2([0,T]\times \p\Omega)$ that satisfies \eqref{eq:assum_2}. Then there holds
	  \begin{align}
\norm{u}_{L^{\infty}([0,T];L^2(\Omega))}+\norm{\p_t u}_{L^{2}([0,T];H^{-1}(\Omega))}&\lesssim \norm{\square u}_{H^{-1}([0,T]\times\Omega)}+\norm{u}_{L^2([0,T]\times\omega)} \nonumber \\
      &+\norm{\compl u}_{L^2([0,T]\times\p\Omega)}.  \label{eq:stab-finite-dim-trace}
    \end{align} 
  \end{theorem}
  \begin{proof}
    We split $u=v+w+\psi$, where
    \begin{equation}
      \begin{cases}
        \square v = \square u, \text{ in }[0,T]\times \Omega,\\
        v|_{[0,T]\times\p\Omega}=\compl u,\\
        v|_{t=0}=0,\ \p_t v|_{t=0}=0,
      \end{cases}
    \end{equation}
    \begin{equation}
      \begin{cases}
        \square w = 0, \text{ in }[0,T]\times \Omega,\\
        w|_{[0,T]\times\p\Omega}=0,\\
        w|_{t=0}=u_{t=0},\ \p_t w|_{t=0}=\p_t u|_{t=0},
      \end{cases}
    \end{equation}
    and 
    \begin{equation}
      \begin{cases}
        \square \psi = 0, \text{ in }[0,T]\times \Omega,\\
        \psi|_{[0,T]\times\p\Omega}=\proj u,\\
        \psi|_{t=0}=0,\ \p_t \psi|_{t=0}=0.
      \end{cases}
    \end{equation}
    The energy estimate \cite[Proposition A.1]{BFMO} for $v$ reads
    \begin{equation}
      \label{eq:estimate_v}
      \norm{v}_{L^\infty([0,T];L^2(\Omega))}+\norm{\p_t v}_{L^2([0,T];H^{-1}(\Omega))}\lesssim \norm{\square u}_{H^{-1}([0,T]\times \Omega)}+\norm{\compl u}_{L^2([0,T]\times\p \Omega)}.
    \end{equation}
    It follows from the geometry control condition and the observability \cite[Theorem A.4]{BFMO} that 
    \begin{align}
      \label{eq:estimate_initial}
      \norm{w}_{L^\infty([0,T];L^2(\Omega))}+\norm{\p_t w}_{L^2([0,T];H^{-1}(\Omega))}&\lesssim \norm{u|_{t=0}}_{L^2(\Omega)}+\norm{\p_t u|_{t=0}}_{H^{-1}(\Omega)}\\
      &\lesssim \norm{w}_{L^2([0,T]\times \omega)}. \nonumber
    \end{align}
    Next we introduce the operator
    \begin{align}
      A:\FiniteDimSpace &\to L^2([0,T]\times \omega),\\
      A(g)&=\varphi|_{[0,T]\times \omega},
    \end{align}
    where $\varphi$ solves
    \begin{equation}
      \begin{cases}
        \square \varphi = 0, \text{ in }[0,T]\times \Omega,\\
        \varphi|_{[0,T]\times\p\Omega}=g,\\
        \varphi|_{t=0}=0,\ \p_t \varphi|_{t=0}=0.
      \end{cases}
    \end{equation}
    According to the unique continuation for wave equation and the assumption \eqref{eq:assum_2}, $A$ is an injective map from a finite dimensional space. Hence there is a norm such that $A$ is an isometry. Since all norms on a finite dimensional space are equivalent, we have
    \begin{equation}
      \norm{g}_{L^2([0,T]\times \p\Omega)}\lesssim \norm{A(g)}_{L^2([0,T]\times \omega)}.
    \end{equation}
    Notice $A(\proj u)=\psi|_\omega$, then the above inequality and the energy estimate for $\psi$ yields
    \begin{equation}
      \label{eq:estimate_P}
      \norm{\psi}_{L^\infty([0,T];L^2(\Omega))}+\norm{\p_t \psi}_{L^2([0,T];H^{-1}(\Omega))}\lesssim \norm{\proj u}_{L^2([0,T]\times \p\Omega)}\lesssim \norm{\psi}_{L^2([0,T]\times\omega)}.
    \end{equation}
    We denote the range of $A$ by $R(A)$. Then we consider the operator 
    \begin{align}
      B:L^2(\Omega)\times H^{-1}(\Omega)&\to L^2([0,T]\times \omega),\\
      B(\phi_0,\phi_1)&=\rho|_{[0,T]\times\omega},
    \end{align}
    where $\rho$ solves 
    \begin{equation}
      \begin{cases}
        \square \rho = 0, \text{ in }[0,T]\times \Omega,\\
        \rho|_{[0,T]\times\p\Omega}=0,\\
        \rho|_{t=0}=\phi_0,\ \p_t \rho|_{t=0}=\phi_1.
      \end{cases}
    \end{equation}
    According the observability \cite[Theorem A.4]{BFMO}, there holds
    \begin{equation}
      \norm{\phi_0}_{L^2(\Omega)}+\norm{\phi_1}_{H^{-1}(\Omega)}\lesssim \norm{\rho}_{L^2([0,T]\times\omega)}.
    \end{equation}
    Furthermore, the regularity estimate for the wave equation reads
    \begin{equation}
      \norm{\rho}_{L^2([0,T]\times\omega)}\lesssim \norm{\phi_0}_{L^2(\Omega)}+\norm{\phi_1}_{H^{-1}(\Omega)}.
    \end{equation}
    Thus the range of $B$, denoted by $R(B)$ is a closed subspace of $L^2([0,T]\times \omega)$. \par
    Let us show that $R(A)\cap R(B)=\{0\}$. Suppose $\theta \in R(A)\cap R(B)$, then there exists $g\in \V$ and $(\phi_0,\phi_1)\in L^2(\Omega)\times H^{-1}(\Omega)$ such that $A(g)=\theta=B(\phi_0,\phi_1)$. 
    That means there are $\varphi_1$, $\varphi_2$ satisfying
    \begin{equation}
      \begin{cases}
        \square \varphi_1 = 0, \text{ in }[0,T]\times \Omega,\\
        \varphi_1|_{[0,T]\times\p\Omega}=g,\\
        \varphi_1|_{t=0}=0,\ \p_t \varphi_1|_{t=0}=0,
      \end{cases}
    \end{equation}
    and 
    \begin{equation}
      \begin{cases}
        \square \varphi_2 = 0, \text{ in }[0,T]\times \Omega,\\
        \varphi_2|_{[0,T]\times\p\Omega}=0,\\
        \varphi_2|_{t=0}=\phi_0,\ \p_t \varphi_2|_{t=0}=\phi_1,
      \end{cases}
    \end{equation}
    respectively, and $\varphi_1|_{[0,T]\times\omega}=A(g)=\theta=B(\phi_0,\phi_1)=\varphi_2|_{[0,T]\times \omega}$.\par
    Then $\varphi=\varphi_1-\varphi_2$ satisfies
    \begin{equation}
      \begin{cases}
        \square \varphi = 0, \text{ in }[0,T]\times \Omega,\\
        \varphi|_{[0,T]\times\p\Omega}=g,\\
        \varphi|_{t=0}=-\phi_0,\ \p_t \varphi|_{t=0}=-\phi_1,
      \end{cases}
    \end{equation}
    with $\varphi|_{[0,T]\times\omega}=0$ by definition. Then according to the unique continuation, there holds $g|_{\Gamma}=0$, and the assumption \eqref{eq:assum_2} gives $g=0$. Hence $\theta=A(g)=0$. \par
    We write $\proj_{B}$ for the projection operator on $R(B)$ and $\compl_B=1-\proj_{B}$ the corresponding complement operator. Then we observe that the restriction of $\compl_B$ on $R(A)$ is injective. To see this, for any $\varphi\in R(A)$ with $\compl_B \varphi=0$, we have $\varphi\in R(B)$. Since $R(A)\cap R(B)=\{0\}$, we can obtain $\varphi=0$. Thus $\compl_B|_{R(A)}$ is an injective map from a finite dimensional space, and there exists a norm such that $\compl_B|_{R(A)}$ is an isometry. Then there holds
    \begin{equation}
      \norm{\varphi}_{L^2([0,T]\times \omega)}\lesssim \norm{\compl_B \varphi}_{L^2([0,T]\times\omega)},\ \varphi\in R(A).
    \end{equation}
    Recall that $\psi|_{[0,T]\times\omega}\in R(A)$ and $\compl_B w=0$ as $w|_{[0,T]\times\omega}\in R(B)$. We have
    \begin{equation}\label{eq:psileqwpluspsi}
      \norm{\psi}_{L^2([0,T]\times \omega)}\lesssim \norm{\compl_B \psi}_{L^2([0,T]\times \omega)}=\norm{\compl_B (\psi+w)}_{L^2([0,T]\times \omega)}\leq \norm{w+\psi}_{L^2([0,T]\times \omega)},
    \end{equation}
    and 
    \begin{equation}\label{eq:wleqwpluspsi}
      \norm{w}_{L^2([0,T]\times \omega)}\leq \norm{w+\psi}_{L^2([0,T]\times \omega)}+\norm{\psi}_{L^2([0,T]\times \omega)}\lesssim \norm{w+\psi}_{L^2([0,T]\times \omega)}.
    \end{equation}
   Using \eqref{eq:psileqwpluspsi}, \eqref{eq:wleqwpluspsi}, the relation $w+\psi = u-v$ and \eqref{eq:estimate_v} we see that
    \begin{multline}\label{eq:wpsilequ}
    \norm{w}_{L^2([0,T]\times\omega)}+\norm{\psi}_{L^2([0,T]\times\omega)} \lesssim \norm{w+\psi}_{L^2([0,T]\times\omega)} 
\leq \norm{u}_{L^2([0,T]\times\omega)}+\norm{v}_{L^2([0,T]\times\omega)}\\ 
\lesssim \norm{u}_{L^2([0,T]\times\omega)}+\norm{\square u}_{H^{-1}([0,T]\times\Omega)}+\norm{\compl u}_{L^2([0,T]\times\p\Omega)}.
    \end{multline}
     Combining now \eqref{eq:estimate_v}, \eqref{eq:estimate_initial}, \eqref{eq:estimate_P} and \eqref{eq:wpsilequ} we conclude that
    \begin{equation}
      \begin{split}
        &\norm{u}_{L^\infty([0,T];L^2(\Omega))}+\norm{\p_t u}_{L^2([0,T];H^{-1}(\Omega))}\lesssim \norm{v}_{L^\infty([0,T];L^2(\Omega))}+\norm{\p_t v}_{L^2([0,T];H^{-1}(\Omega))}\\
      &+\norm{w}_{L^\infty([0,T];L^2(\Omega))}+\norm{\p_t w}_{L^2([0,T];H^{-1}(\Omega))}+\norm{\psi}_{L^\infty([0,T];L^2(\Omega))}+\norm{\p_t \psi}_{L^2([0,T];H^{-1}(\Omega))}\\
      &\lesssim \norm{\square u}_{H^{-1}([0,T]\times\Omega)}+\norm{\compl u}_{L^2([0,T]\times\p\Omega)}+\norm{w}_{L^2([0,T]\times\omega)}+\norm{\psi}_{L^2([0,T]\times\omega)}\\
      &\lesssim \norm{\square u}_{H^{-1}([0,T]\times\Omega)}+\norm{\compl u}_{L^2([0,T]\times\p\Omega)}+\norm{u}_{L^2([0,T]\times\omega)}.
      \end{split}
    \end{equation}
    This proves the claim.
  \end{proof}
   We remark that assumption \eqref{eq:assum_2} is necessary for Theorem \ref{thm:lipschitz-stab-fin}. Indeed, we show that if there is a non-zero $f\in \FiniteDimSpace\subset L^2([0,T]\times \p\Omega)$ with $f|_\Gamma = 0$, then we can find a nonzero $u$ such that the right-hand side of \eqref{eq:stab-finite-dim-trace} vanishes.
    
    We denote $[0,T]\times \p\Omega \setminus\Gamma$ by $\Upsilon$ and we decompose $\Upsilon= \Upsilon_-\cup \Upsilon_+$, where 
    \begin{align}
      \Upsilon_+ & := \{(t,x)\in [0,T]\times\p\Omega\mid T/2\leq t\leq T,\ \dist(x,\omega)> T-t\},\\
      \Upsilon_- & := \{(t,x)\in [0,T]\times\p\Omega\mid 0\leq t< T/2,\ \dist(x,\omega)> t\}.
    \end{align}
    We write $f_+ = f|_{\Upsilon_+}$ and $f_- = f|_{\Upsilon_-}$. Let $u_+$ satisfy
    \begin{equation}
      \begin{cases}
      \square u_+(t,x) = 0, \text{ in }[0,T]\times \Omega,\\
      u_+(t,\cdot)|_{[0,T]\times\p\Omega}=f_+(t),\\
      u_+(0,\cdot)=0,\ \p_t u_+(0,\cdot) = 0.
    \end{cases}
    \end{equation}
    According to the finite speed of propagation, there holds $u_+|_{[0,T]\times\omega}=0$. Let $v$ satisfy the backward wave equation, that is
    \begin{equation}
      \begin{cases}
      \square v(-t,x) = 0, \text{ in }[0,T]\times \Omega,\\
      v(t,\cdot)|_{[0,T]\times\p\Omega}=f_-(t),\\
      v(T,\cdot)=0,\ \p_t v(T,\cdot) = 0.
    \end{cases}
    \end{equation}
    Using finite speed of propagation again, we have $v|_{[0,T]\times \omega}=0$. Due to the energy estimate, $v(0,\cdot)\in L^2(\Omega)$ and $\p_t v(0,\cdot)\in H^{-1}(\Omega)$. Write $u_-(t)=v(T-t)$, then $u_-$ solves
    \begin{equation}
      \begin{cases}
      \square u_-(t,x) = 0, \text{ in }[0,T]\times \Omega,\\
      u_-(t,\cdot)|_{[0,T]\times\p\Omega}=f_-(t),\\
      u_-(0,\cdot)=v(0,\cdot),\ \p_t u_-(0,\cdot) = \p_t v(0,\cdot).
    \end{cases}
    \end{equation}
    Finally, let $u=u_+ + u_-$, it follows that
    \begin{equation}
      \begin{cases}
      \square u(t,x) = 0, \text{ in }[0,T]\times \Omega,\\
      u(t,\cdot)|_{[0,T]\times\p\Omega}=f(t),\\
      u(0,\cdot)=v(0,\cdot),\ \p_t u(0,\cdot) = \p_t v(0,\cdot).
    \end{cases}
  \end{equation}
  as well as $u|_{[0,T]\times \omega}=0$. Moreover, since $f\in \FiniteDimSpace$, the right hand side of \eqref{eq:stab-finite-dim-trace} vanishes. Hence we conclude that \eqref{eq:assum_2} is necessary.

\section{Finite element method}\label{section:FEM}
For the numerical discretization we will employ a stabilized space-time dG method that has been introduced in \cite{BP24_wave} for problem \eqref{eq:PDE+data_constraint} with the additional constraint $u = 0$ on $\Sigma$. Here we adapt this method to fit the case where no boundary data is known. Whenever possible we will make use of results already established in \cite{BP24_wave} to keep the presentation brief and focus on the novel aspects required to adjust the method to the stability etimates from Section \ref{section:stab}.

\subsection{Discretization of the space-time cylinder}\label{ssection:discr-geom}
\subsubsection{Spatial mesh}
Notice that in the geometrical configuration considered in Section \ref{ssection:stab-hoelder} neither $\Omega$ nor $\omega$ is a polygonal domain since part of their boundary is given by a sphere. We fit this boundary by allowing the elements to be curved. We start with a quasi-uniform triangulation $\hat{\mathcal{T}}_h$ of mesh width $h$ such that $\hat{\Omega}:= \cup_{K \in \hat{\mathcal{T}}_h} K $ contains $\Omega$, i.e.\ $\Omega \subset \hat{\Omega}$. Then we set $\mathcal{T}_h := \{ K \cap \Omega \mid K \in  \hat{\mathcal{T}}_h \}$, where we note that $\Omega = \cup_{K \in \mathcal{T}_h} K$ holds true. It can be ensured, see \cite[Appendix B]{BFMO}, that the continuous trace inequality
\bel{ieq:cont-trace-ieq}
\norm{v}_{ [L^2(\partial K)]^d } \lesssim \left( h^{-1/2} \norm{v}_{ [L^2(K)]^d } + h^{1/2} \norm{ \nabla v  }_{ [L^2(K)]^d }   \right), \; \forall v \in [H^1(K)]^d, \; K \in \mathcal{T}_h
\ee
holds. 

\subsubsection{Partition of the time axis} 
Let $(t_0, t_N) = (-T,T)$ or $(t_0, t_N) = (0,T)$ and consider a partition $t_0 < t_1 < \ldots < t_N$, of the time axis into subintervals $I_n := (t_n,t_{n+1}), n=0,\ldots,N-1$ of equal length $ \abs{t_{n+1} - t_n} \sim h$.
A corresponding partition of the space-time cylinder into time slabs is given by
\bel{eq:def-time-slab}
Q^n := I_n \times \Omega, \; \Sigma^n = I_n \times \partial \Omega, \; n=0,\ldots,N-1, \quad Q := \cup_{n=0}^{N-1} Q^n, \; \Sigma := \cup_{n=0}^{N-1} \Sigma^n.
\ee
We introduce a notation for space-time integrals on the slabs
\[ \begin{array}{r}
(u,v)_{Q^n} := \int\limits_{I_n} \int\limits_{\Omega} u v \; \dX \; \dT, 
\;
a(u,v)_{Q^n} := \! \int\limits_{I_n} \int\limits_{\Omega} \nabla u \cdot \nabla v \; \dX \; \dT,
\;
(u,v)_{ \Sigma^n }\\[3mm] = \int\limits_{I_n}  \int\limits_{ \partial \Omega} u v \; \dS \; \dT 
\end{array}
\]
and we set $\norm{v}^2_{Q^n} := (v,v)_{Q^n}$ and $ \norm{v}^2_{\Sigma^n} = (v,v)_{\Sigma^n}$. 
The data domain and integrals thereon are defined similarly
\[
\omega^n := I_n \times \omega, \quad (u,v)_{\omega^n} := \int\limits_{I_n} \int\limits_{\omega} u v \; \dX \; \dT,   \quad \norm{ u }_{ \STdata }^2 = \sum\limits_{n=0}^{N-1}(u,v)_{\omega^n}.
\]

\subsubsection{Finite element spaces} 

Based on the mesh $\mathcal{T}_h$ we introduce the (spatial) finite element space
\[
V_h^k := \{ u \in H^1(\Omega) \mid u|_K \in \mathbb{P}_k(K) \; \forall K \in \mathcal{T}_h \},
\]
where $\mathbb{P}_k(K)$ denotes the spaces of polynomials of order at most $k \in \mathbb{N}$ on $K$. 
To discretize time, let $\mathbb{P}_q(I_n)$ denote the space of polynomials up to degree $q \in \mathbb{N}_0$ on $I_n$ and define a discontinuous (in-time) finite element space by 
\[ \FullyDiscrSpace{k}{q} :=  \otimes_{n=0}^{N-1} \mathbb{P}_q(I_n) \otimes V_h^k,
\quad \ProdFullyDiscrSpace{k}{q} := \FullyDiscrSpace{k}{q} \times \FullyDiscrSpace{k}{q},  
\quad q \in \mathbb{N}_{0}, \; k \in \mathbb{N}. \]
Moreover, for tuples $\underline{\mathbf{U_h}} \in \ProdFullyDiscrSpace{k}{q}$ we introduce the notation
\bel{eq:mixed-vars} 
  \underline{\mathbf{U_h}} := ( \underline{u}_1,\underline{u}_2), \quad \underline{u}_j \in \FullyDiscrSpace{k}{q}, \; j= 1,2.
\ee
Functions in $ \ProdFullyDiscrSpace{k}{q} $ are potentially discontinuous between time-slabs. To denote the temporal jumps, we employ the common notation  
\bel{eq:def-jump-in-time}
v^{n}_{\pm}(x) := \lim_{s \rightarrow 0^{+}} v(x,t_n \pm s), \text{ and } \jump{ v^n } := v_{+}^n - v_{-}^n.
\ee

\subsection{A stabilized primal dual FEM using dG in time}\label{ssection:fully-discrete-method}
We proceed with the discussion of the variational formulation for our problem. 
The PDE constraint will be enforced using the bilinear form
\begin{equation}\label{eq:bil-form-def-fully-discrete}
A[ \underline{\mathbf{U}}_h, \underline{\mathbf{Y}}_h] = \! \sum\limits_{n=0}^{N-1} \! \left\{ (\partial_t \underline{u}_2, \underline{y}_1)_{Q^n} + a(\underline{u}_1,\underline{y}_1)_{Q^n}  
    + (\partial_t \underline{u}_1 - \underline{u}_2, \underline{y}_2)_{Q^n}    
  - ( \nabla \underline{u}_1 \! \cdot \! \mathbf{n}, \underline{y}_1)_{ \Sigma^n } 
  \right\}.
\end{equation}
Using this bilinear form we will formulate problem (\ref{eq:PDE+data_constraint}) at the discrete level as a PDE constrained optimization problem. 
In applications, the data to be fitted is usually corrupted by measurement noise. Here we assume that instead of the exact data $u_{\omega} = u$ on $\STdata$ we are merely given perturbed data 
\bel{eq:noise}
\tilde{u}_{\omega} = u_{\omega} + \delta u \quad \text{in } \STdata, \qquad \delta u \in L^2(\STdata).
\ee
We will search for approximate solutions $(\underline{\mathbf{U}}_h,\underline{\mathbf{Z}}_h) \in \ProdFullyDiscrSpace{k}{q} \times \ProdFullyDiscrSpace{k_{\ast}}{q_{\ast}}$ of \eqref{eq:PDE+data_constraint} as stationary points of the Lagrangian
\begin{align}
\mathcal{L}( \underline{\mathbf{U}}_h,\underline{\mathbf{Z}}_h) :&= 
 \frac{1}{2} \norm{ \underline{u}_1 - \tilde{u}_{\omega}}_{\omega_T}^2 +  A[ \underline{\mathbf{U}}_h, \underline{\mathbf{Z}}_h ]  
	 \nonumber  \\
  & + \frac{1}{2} S_h( \underline{\mathbf{U}}_h, \underline{\mathbf{U}}_h) + \frac{1}{2} S^{\uparrow \downarrow}_{h}( \underline{\mathbf{U}}_h, \underline{\mathbf{U}}_h ) - \frac{1}{2} S_h^{\ast}( \underline{\mathbf{Z}}_h, \underline{\mathbf{Z}}_h). \quad \label{eq:Lagrangian}
\end{align}
The terms in the first row enforce the data respectively PDE constraint, while the terms in the second row act as regularizers which will be described in detail in 
Section \ref{ssection:stab}. 
The first order optimality conditions then take the form:
Find $(\underline{\mathbf{U}}_h, \underline{\mathbf{Z}}_h) \in \ProdFullyDiscrSpace{k}{q} \times \ProdFullyDiscrSpace{ k_{\ast} }{ q_{\ast} } $ such that
\bel{eq:opt_compact_fully_disc} 
B[ ( \underline{\mathbf{U}}_h,\underline{\mathbf{Z}}_h), ( \underline{\mathbf{W}}_h ,\underline{\mathbf{Y}}_h ) ] = (\tilde{u}_{\omega}, \underline{w}_1)_{\STdata }
\ee 
for all $( \underline{\mathbf{W}}_h ,\underline{\mathbf{Y}}_h ) \in \ProdFullyDiscrSpace{k}{q} \times \ProdFullyDiscrSpace{ k_{\ast} }{ q_{\ast} } $, 
where 
\begin{align}\label{eq:complete_bfi_def_fully_discr}
B[ ( \underline{\mathbf{U}}_h,\underline{\mathbf{Z}}_h), ( \underline{\mathbf{W}}_h ,\underline{\mathbf{Y}}_h ) ]  := & (\underline{u}_1,\underline{w}_1)_{\STdata} + A[\underline{\mathbf{W}}_h, \underline{\mathbf{Z}}_h ] + A[\underline{\mathbf{U}}_h, \underline{\mathbf{Y}}_h ] \nonumber \\
	& + S_h( \underline{\mathbf{U}}_h, \underline{\mathbf{W}}_h ) + \StabTimeJumps{ \underline{\mathbf{U}}_h}{ \underline{\mathbf{W}}_h ) } - S_h^{\ast}( \underline{\mathbf{Y}}_h, \underline{\mathbf{Z}}_h). 
\end{align}

\subsubsection{Stabilization}\label{ssection:stab}
Since we are dealing with a potentially ill-posed problem, suitable regularization terms have to be introduced in addition to the PDE constraint. The choice of these terms is driven by the error analysis presented in Section \ref{section:error-analysis}. The main idea to obtain an upper bound on the error $e_h := u - \underline{u}_1$, where $u$ denotes the exact solution of (\ref{eq:PDE+data_constraint}) and $\underline{u}_1$ is the first component of $\underline{\mathbf{U}}_h$, is to apply one of the stability estimates (\ref{eq:Hoelder-stab}) or (\ref{eq:stab-finite-dim-trace}) to\footnote{To be precise, we will actually apply these estimates to a smoothed version of $e_h$ since $\underline{u}_1$ is not continuous in time.} $e_h$. It then remains to control the terms on the right hand side of these estimates by means of the stabilization. 

The following series of terms is dedicated to deal with the $H^{-1}$-norm of the PDE residual which appears in both estimates. We define a spatial continuous interior penalty term $J(\cdot,\cdot)$, which penalizes the jump of the gradient $\jump{ \nabla \underline{u}_1 }_F$ over the interior facet set $\mathcal{F}_i$ of $\mathcal{T}_h$, a Galerkin-least squares term $G(\cdot,\cdot)$ which enforces the PDE at the element level and a term $I_{0}(\cdot,\cdot)$ which ensures that $\underline{u}_2 = \partial_t \underline{u}_1$.
\begin{align*}
& \StabCIP{ \underline{\mathbf{U}}_h}{ \underline{\mathbf{W}}_h } := \sum\limits_{n=0}^{N-1} \int\limits_{I_n} \sum\limits_{F \in \mathcal{F}_i} h  ( \jump{ \nabla \underline{u}_1 }_F , \jump{ \nabla \underline{w}_1 }_F  )_{F} \; \dT,  
\\ 
& \StabGLS{ \underline{\mathbf{U}}_h}{ \underline{\mathbf{W}}_h } := \sum\limits_{n=0}^{N-1} \int\limits_{I_n} \sum\limits_{K \in \mathcal{T}_h }  h^2 ( \partial_t \underline{u}_2 - \Delta \underline{u}_1 , \partial_t \underline{w}_2 - \Delta \underline{w}_1 )_{K}  \; \dT, \quad   
\\
& \StabDt{ \underline{\mathbf{U}}_h}{ \underline{\mathbf{W}}_h } 
	:= \sum\limits_{n=0}^{N-1}(\underline{u}_2 - \partial_t \underline{u}_1, \underline{w}_2 - \partial_t \underline{w}_1)_{Q^n}.
\end{align*}

Additionally, we require a stabilization that will affect the communication between different time-slabs in a similar way as the term $J(\cdot,\cdot)$ affects the coupling between spatial elements. In the end, this requirement stems from the need to enforce that 
$ \underline{u}_1$ is globally in $H^1$ in order to be allowed to apply the stability 
estimates (\ref{eq:Hoelder-stab}) or (\ref{eq:stab-finite-dim-trace}) to the error. We introduce   
\begin{equation}\label{eq:def-jump-stab}
\StabTimeJumps{ \underline{\mathbf{U}}_h }{ \underline{\mathbf{W}}_h } 
:= \underline{I}_1^{\uparrow \downarrow}( \underline{\mathbf{U}}_h, \underline{\mathbf{W}}_h ) + 
\underline{I}_2^{\uparrow \downarrow}( \underline{\mathbf{U}}_h, \underline{\mathbf{W}}_h ), 
\end{equation} 
\begin{align*}
\underline{I}_1^{\uparrow \downarrow}( \underline{\mathbf{U}}_h, \underline{\mathbf{W}}_h ) &:= \sum\limits_{n=1}^{N-1} \left\{ \frac{1}{h} ( \jump{ \underline{u}_1^n } , \jump{ \underline{w}_1^n } )_{ \Omega} 
	+ h ( \jump{ \nabla \underline{u}_1^n } , \jump{ \nabla \underline{w}_1^n } )_{ \Omega} \right\}, \\
\underline{I}_2^{\uparrow \downarrow}( \underline{\mathbf{U}}_h, \underline{\mathbf{W}}_h ) &:= \sum\limits_{n=1}^{N-1} \frac{1}{h} ( \jump{ \underline{u}_2^n } , \jump{ \underline{w}_2^n } )_{ \Omega}.
\end{align*}
While the previous terms are required in either regime of stability, we now 
introduce two terms that are specific to each case. 
The estimate (\ref{eq:Hoelder-stab}) is conditional in the sense that $\norm{u}_{L^2(\Qmin)}^{1-\alpha}$ appears on the right hand side. In order to control the norm of $\underline{u}_1$ on $\STdom$ we introduce a Tikhonov term 
\begin{equation*}
\StabTikh{ \underline{\mathbf{U}}_h}{ \underline{\mathbf{W}}_h } := \gamma \sum\limits_{n=0}^{N-1} h^{2s} ( \underline{u}_1, \underline{w}_1)_{ Q^n },
\end{equation*}
where $\gamma \geq 0$ is a constant and $s := \min{ \{ q,k \} }$.  Notice that this regularization vanishes suitably 
fast with $h$ which ensures that it will not spoil the error estimates in case the data are clean.  \par 
Choosing $\gamma = 0$ is allowed  when we know a finite dimensional space in which the trace of $u$ on $\Sigma$ is contained. In this case, the Tikhonov regularization is in fact obsolete since estimate (\ref{eq:stab-finite-dim-trace}) contains no mass term in the volume. 
Instead, we need to control the distance of the trace of $\underline{u}_1$ to the finite dimensional subspace on $\Sigma$. 
To this end, we introduce 
\begin{equation}\label{eq:stab-compl}
\StabTrace{ \underline{\mathbf{U}}_h}{ \underline{\mathbf{W}}_h } := ( \compl \underline{u}_1, \compl \underline{w}_1)_{ \Sigma },
\end{equation}
where we recall that $\compl = 1 - \proj$ for some projection $\proj$ on $\FiniteDimSpace \subset L^2(\Sigma)$. We refer to Section \ref{ssection:impl-trace-stab} for a discussion on how to realize (\ref{eq:stab-compl}) conveniently in an implementation.
If no finite dimensional subspace $\FiniteDimSpace$ is known in which the trace of the solution is contained, we will set $\FiniteDimSpace =  L^2(\Sigma)$ which implies that $\StabTrace{ \underline{\mathbf{U}}_h}{ \underline{\mathbf{W}}_h } = 0$  for all $(\underline{\mathbf{U}}_h, \underline{\mathbf{W}}_h) \in \ProdFullyDiscrSpace{k}{q} \times \ProdFullyDiscrSpace{k}{q}$ holds. This convention allows us to treat both stability scenarios discussed in Section \ref{section:stab} at the same time.

Finally, let us collect the slab-local parts of the primal stabilization:
\begin{align}
S_{h}( \underline{\mathbf{U}}_h, \underline{\mathbf{W}}_h ) :=& 
	\StabCIP{ \underline{\mathbf{U}}_h}{ \underline{\mathbf{W}}_h } 
	+ \StabGLS{ \underline{\mathbf{U}}_h}{ \underline{\mathbf{W}}_h }
    + \StabDt{ \underline{\mathbf{U}}_h}{\underline{\mathbf{W}}_h } \nonumber \\
	& + \StabTikh{ \underline{\mathbf{U}}_h}{ \underline{\mathbf{W}}_h  }
          + \StabTrace{ \underline{\mathbf{U}}_h}{ \underline{\mathbf{W}}_h }. \label{eq:Def-Sh}
\end{align}
It remains to introduce the stabilization of the dual variable $ \underline{\mathbf{Z}}_h$. To this end,  we define:
\begin{equation}\label{eq:dual_stab_fully_disc}
S^{\ast}_h( \underline{\mathbf{Y}}_h, \underline{\mathbf{Z}}_h) := 
	 \sum\limits_{n=0}^{N-1} \big\{ (\underline{y}_1,\underline{z}_1)_{Q^n} + a(\underline{y}_1,\underline{z}_1)_{Q^n} + (\underline{y}_2,\underline{z}_2)_{Q^n} + h^{-1} (\underline{y}_1 ,\underline{z}_1 )_{ \Sigma^n }  \big\}. 
\end{equation}

\subsubsection{Remarks on the implementation}\label{ssection:impl-trace-stab}
To implement \eqref{eq:stab-compl} it is convenient to introduce an additional variable. 
Let  $ B^{0}[ \cdot ,\cdot] $ denote the bilinear form $ B[ \cdot ,\cdot] $ with the term $\StabTrace{ \cdot }{ \cdot }$ removed. 
We then solve the augmented variational formulation:

Find $(\underline{\mathbf{U}}_h, \underline{\mu_h} , \underline{\mathbf{Z}}_h) \in \ProdFullyDiscrSpace{k}{q} \times \FiniteDimSpace  \times \ProdFullyDiscrSpace{ k_{\ast} }{ q_{\ast} } $ such that
\bel{eq_var_formu_augmented} 
B^{\#}[ ( \underline{\mathbf{U}}_h, \underline{\mu_h}, \underline{\mathbf{Z}}_h), ( \underline{\mathbf{W}}_h , \underline{\eta_h},   \underline{\mathbf{Y}}_h ) ] = (\tilde{u}_{\omega}, \underline{w}_1)_{\STdata }
\ee 
for all $( \underline{\mathbf{W}}_h , \underline{\eta_h}, \underline{\mathbf{Y}}_h ) \in \ProdFullyDiscrSpace{k}{q} \times \FiniteDimSpace \times \ProdFullyDiscrSpace{ k_{\ast} }{ q_{\ast} } $, 
where
\begin{align*}
B^{\#}[ ( \underline{\mathbf{U}}_h, \underline{\mu_h}, \underline{\mathbf{Z}}_h), ( \underline{\mathbf{W}}_h , \underline{\eta_h},   \underline{\mathbf{Y}}_h ) ] 
:= B^{0}[ ( \underline{\mathbf{U}}_h,\underline{\mathbf{Z}}_h), ( \underline{\mathbf{W}}_h ,\underline{\mathbf{Y}}_h ) ] + ( \underline{u}_1 - \underline{\mu_h}, \underline{w}_1 - \underline{\eta_h} )_{\Sigma}. 
\end{align*}
Note that $( \underline{u}_1 - \underline{\mu_h}, - \underline{\eta_h} )_{\Sigma} = 0$ for all $\underline{\eta_h} \in  \FiniteDimSpace$ implies that  $\underline{\mu_h} = \proj \underline{u}_1 |_{\Sigma}$. Hence, the variable $\underline{\mu_h}$ can in fact be eliminated to arrive at the variational formulation \eqref{eq:opt_compact_fully_disc} which is more suitable for the error analysis.

To solve the arising linear systems efficiently, we actually implement a slightly modified version of \eqref{eq_var_formu_augmented}. In \cite[Section 5.1]{BP24_wave} an iterative solution strategy has been suggested for space-time dG discretizations of unique continuation problems that relies on the assumption that the discretization spaces are fully discontinuous between the time slabs and continuity in between is enforced via suitable jump terms. Hence, to allow us to use this strategy we replace the global space $\FiniteDimSpace$ in \eqref{eq_var_formu_augmented} by its 'broken-in-time' version $ \mathcal{V}^{\mathrm{disc}} :=  \otimes_{n=0}^{N-1}  \; \mathcal{V}_M|_{Q^n }$ and add the additional term 
\[
\sum\limits_{n=1}^{N-1} \frac{1}{h} ( \jump{ \underline{\mu}_h^n }   , \jump{ \underline{\eta}_h^n } )_{\partial \Omega}
\]
to the variational formulation to recover continuity in time. After these modifications the solution strategy discussed in \cite[Section 5.1]{BP24_wave} can be applied in a straightforward manner.


\subsection{Stability, continuity and interpolation results}\label{ssection:stab-interp}
 In this section we collect some elementary results to prepare the error analysis of the finite element method introduced in Section~\ref{ssection:fully-discrete-method}. The first step is to show that the bilinear form $B[ \cdot, \cdot ]$ is inf-sup stable with respect to a suitable norm.
Let us define
\begin{equation}\label{eq:def-h-semi-norms}
\abs{ \underline{\mathbf{U}}_h }_{S_h} := S_h( \underline{\mathbf{U}}_h, \underline{\mathbf{U}}_h)^{1/2}, \quad
	\abs{ \underline{\mathbf{U}}_h   }_{ \uparrow \downarrow } := \StabTimeJumps{ \underline{\mathbf{U}}_h}{ \underline{\mathbf{U}}_h }^{1/2}, \quad
	\norm{ \underline{\mathbf{Z}}_h }_{S_h^{\ast}} := S_h^{\ast}( \underline{\mathbf{Z}}_h, \underline{\mathbf{Z}}_h)^{1/2} 
\end{equation}
and
\bel{eq:triple_norm_h_def}
\tnorm{ (\underline{\mathbf{U}}_h, \underline{\mathbf{Z}}_h) }^2  :=  \abs{ \underline{\mathbf{U}}_h }_{S_h}^2 + \abs{ \underline{\mathbf{U}}_h   }_{ \uparrow \downarrow }^2 + \norm{ \underline{u}_1 }^2_{\STdata} +  \norm{ \underline{\mathbf{Z}}_h }_{S_h^{\ast}}^2. 
\ee
To prove that this defines a norm on the discrete spaces, we require the following result which follows easily from integration by parts.
\begin{lemma}\label{lem:IP-norm}
For $\underline{\mathbf{U}}_h \in \ProdFullyDiscrSpace{k}{q}$
and $y \in H^1_0(\STdom)$ it holds that:
\begin{align}
& \int\limits_{\STdom} \left\{ -(\partial_t \underline{u}_1)  \partial_t y + \nabla \underline{u}_1 \nabla y \right\} 
 =  \sum\limits_{n=0}^{N-1} \int\limits_{I_n} \sum\limits_{K \in \mathcal{T}_h }  ( \partial_t \underline{u}_2 - \Delta \underline{u}_1 , y )_{K} \; \dT \label{eq:IP-H-1} \\ 
	& + \sum\limits_{n=0}^{N-1} ( \underline{u}_2 - \partial_t \underline{u}_1,  \partial_t y )_{Q^n}  + \sum\limits_{n=0}^{N-1} \int\limits_{I_n} \sum\limits_{F \in \mathcal{F}_i}  ( \jump{ \nabla \underline{u}_1 \cdot \mathbf{n} }_{F} , y  )_{F} \; \dT 
	+ \sum\limits_{n=1}^{N-1}(  \jump{ \underline{u}_2^n } , y )_{ \Omega}, \nonumber \end{align}
where $ \jump{ \nabla \underline{u}_1 \cdot \mathbf{n} }_{F} $ denotes the jump of the normal derivative over interior facets.
\end{lemma}

\begin{lemma}[Norm]\label{lem:tnorm}
If either $\gamma >0$ or $\FiniteDimSpace$ is finite dimensional, then $\tnorm{ (\cdot, \cdot) }$ defines a norm on $\ProdFullyDiscrSpace{k}{q} \times \ProdFullyDiscrSpace{ k_{\ast} }{ q_{\ast} } $. 
\end{lemma}
\begin{proof}
Let $\tnorm{ (\underline{\mathbf{U}}_h, \underline{\mathbf{Z}}_h) } = 0$. It follows readily from the definition of $S_h^{\ast}(\cdot,\cdot)$ that $\mathbf{Z}_h = 0$ and from the definition of $\StabDt{ \cdot }{ \cdot }$ that $\underline{u}_2 = \partial_t \underline{u}_1$. Thus, it remains to show that $\underline{u}_1 = 0$. We now consider the two cases separately. \par 
If $\gamma >0$, then $\underline{u}_1 = 0$ immediately follows from $\StabTikh{ \underline{\mathbf{U}}_h}{ \underline{\mathbf{U}}_h  } = 0$.
On the other hand, if $\FiniteDimSpace$ is finite dimensional and $\gamma =0$ we require another argument. From 
\[
\abs{ \underline{\mathbf{U}}_h }_{S_h} +  \abs{ \underline{\mathbf{U}}_h   }_{ \uparrow \downarrow }  = 0
\]
it follows that $\underline{u}_1 \in H^1(\STdom)$ and that all terms on the right hand side of (\ref{eq:IP-H-1}) vanish. This implies that 
\[
\norm{\Box \underline{u}_1}_{ H^{-1}(\STdom)} = \sup_{\substack{  y \in H^1_0(\STdom), \\ \norm{y}_{H^1(\STdom) } = 1  }}  \int\limits_{\STdom} \left\{ -(\partial_t \underline{u}_1)  \partial_t y + \nabla \underline{u}_1 \nabla y \right\} = 0.  
\]
Hence, applying the stability estimate from (\ref{eq:stab-finite-dim-trace}) to $\underline{u}_1$ yields 
\begin{align*}
	\norm{ \underline{u}_1 }_{L^{\infty}([0,T];L^2(\Omega))}+\norm{\p_t \underline{u}_1  }_{L^{2}([0,T];H^{-1}(\Omega))}&\lesssim \norm{\square \underline{u}_1 }_{H^{-1}(\STdom)}+\norm{ \underline{u}_1 }_{\STdata} + \norm{Q \underline{u}_1}_{L^2( \Sigma)} = 0,
    \end{align*}
where the last two terms vanish due to $\tnorm{ (\underline{\mathbf{U}}_h, \underline{\mathbf{Z}}_h) } = 0$. Hence, $\underline{u}_1 = 0$ as required.
\end{proof}

From the identity 
 $B[ (\underline{\mathbf{U}}_h,\underline{\mathbf{Z}}_h), ( \underline{\mathbf{U}}_h ,-\underline{\mathbf{Z}}_h ) ]  = \tnorm{ (\underline{\mathbf{U}}_h, \underline{\mathbf{Z}}_h) } \tnorm{ (\underline{\mathbf{U}}_h, -\underline{\mathbf{Z}}) }$
we immediately obtain the discrete stability  
\bel{eq:inf-sup}
\sup_{ (\underline{\mathbf{W}}_h,\underline{\mathbf{Y}}_h) \in  \ProdFullyDiscrSpace{k}{q} \times \ProdFullyDiscrSpace{ k_{\ast} }{ q_{\ast} }    } \!\!\!\!\!\!\!\! \frac{ B[ (\underline{\mathbf{U}}_h,\underline{\mathbf{Z}}_h), ( \underline{\mathbf{W}}_h , \underline{\mathbf{Y}}_h ) ] }{  \tnorm{ (\underline{\mathbf{W}}_h, \underline{\mathbf{Y}}_h) } }  \gtrsim \tnorm{ ( \underline{\mathbf{U}}_h, \underline{\mathbf{Z}}_h) }. 
\ee 
In the next lemma we record some continuity results for the bilinear form $A[\cdot,\cdot]$ that were shown in \cite[Lemma 4]{BP24_wave}.
\begin{lemma}[Continuity]\label{lem:bfi_stab_est}
\begin{enumerate}[label=(\alph*)]
\item For $\underline{\mathbf{U}}_h \in \ProdFullyDiscrSpace{k}{q}$ and $ \mathbf{Y} \in [H^1(\STdom)]^2 + \ProdFullyDiscrSpace{q_{\ast} }{k_{\ast} } $ we have 
\[ A[\underline{\mathbf{U}_h}, \mathbf{Y} ] \lesssim \abs{ \underline{\mathbf{U}}_h }_{\underline{S}_h} \left\{ \sum\limits_{n=0}^{N-1} h^{-2} \norm{ y_1 }_{Q^n}^2 +  \norm{ \nabla y_1 }_{Q^n}^2 + \norm{ y_2 }_{Q^n}^2  \right\}^{1/2}. \] 
\item For $\mathbf{U}|_{Q^n} \in \left[ H^{1}(Q^n) \cap L^2(0,T;H^2(\mathcal{T}_h)) \right] \times H^{1}(Q^n) $ for all $n=0,\ldots,N-1$ and $\underline{\mathbf{Y}}_h \in \ProdFullyDiscrSpace{ k_{\ast } }{ q_{\ast} }$ it holds that
\begin{align*}
A[ \mathbf{U}, \underline{\mathbf{Y}}_h ] \lesssim & \bigg( \sum\limits_{n=0}^{N-1} \big\{ \norm{\partial_t u_2}_{Q^n}^2  + \norm{ \nabla u_1 }_{Q^n}^2
 + \norm{\partial_t u_1}_{Q^n}^2 + \norm{ u_2}_{Q^n}^2  \\
 & + \int\limits_{I_n} \sum\limits_{K \in \mathcal{T}_h } h^2 \norm{u_1}_{H^2(K)}^2 \; \dT \big\}  \bigg)^{1/2}  \norm{  \underline{\mathbf{Y}}_h }_{ S^{\ast}_h  }.
\end{align*}
\end{enumerate}
\end{lemma}
Moreover, we will need the following interpolation results from \cite{BP24_wave}. 
\begin{lemma}[Interpolation]\label{lem:interp}
There exists an interpolation operator $\Pi_h = \Pi_{h}^{k,q}$ into $\FullyDiscrSpace{k}{q}$ such that for $s = \min{ \{ q,k \} }$ and $n=0,\ldots,N-1$ 
the following interpolation results are valid.
\begin{enumerate}[label=(\alph*)]
\item $\sum\limits_{n=0}^{N-1}  \{ h^{-1} \norm{u - \Pi_h u }_{ L^2(Q^n) } + \norm{ u - \Pi_h u }_{ H^1(Q^n) } \} \lesssim h^{\ell} \norm{u}_{H^{\ell + 1 }(Q)}$, \\ for $0 \leq \ell \leq s$.
\item For $u \in H^{\ell+1}(Q^n) \cap C^0(I_n, H^2(\Omega))$  have: 
	\[ \left(\sum\limits_{n=0}^{N-1} \int\limits_{I_n} \sum\limits_{K \in \mathcal{T}_h } h^2 \norm{ u - \Pi_h  u }_{ H^2(K) }^2 \; \dT \right)^{1/2} \lesssim h^{\ell} \norm{u}_{H^{\ell +1}(Q)}, \; \text{for } 1 \leq \ell \leq s. \] 
\item  $ \norm{(u - \Pi_h u )( \tau, \cdot )}_{\Omega} \lesssim  h^{\ell + 1/2} \norm{u}_{H^{ \ell + 1 }(Q^n)},  \; \tau \in \{ t_n,t_{n+1}\}$ and $0 \leq \ell \leq s$.
\item $  \norm{\nabla(u - \Pi_h u )( \tau, \cdot )}_{\Omega} \lesssim  h^{\ell - 1/2} \norm{u}_{H^{ \ell + 1 }(Q^n)},  \; \tau \in \{ t_n,t_{n+1}\}$ and $ 1 \leq \ell \leq s$.
\end{enumerate}
\end{lemma}
Interpolation into the spaces containing the dual variable is denoted by $\Pi_h^{\ast} = \Pi_{h}^{k_{\ast},q_{\ast} } $. 
These operators are extended to tuples $\mathbf{W} = (w_1,w_2)$ in a componentwise-manner, i.e.\
$\mathbf{\Pi}_h \mathbf{W} :=  ( \Pi_h w_1 , \Pi_h w_2 )$. 
The following lemma studies the interplay between interpolation of the exact solution of \eqref{eq:PDE+data_constraint} and the stabilization terms. 
\begin{lemma}\label{lem:interp_in_stab}
Let $u \in H^{s+2}(Q)$ for $s = \min{ \{ q,k \} }$ solve \eqref{eq:PDE+data_constraint} and set $ \mathbf{U} = (u, \partial_t u)$.
\begin{enumerate}[label=(\alph*)]
\item For  $u \in H^{s+2 }(Q) \cap C^{0}([0,T],H^2(\Omega))$ we have 
	\[ \abs{ \mathbf{\Pi}_h \mathbf{U} }_{ S_h} \lesssim  h^{ s }  \norm{u}_{ H^{ s+1 }(\STdom)  }  \text{ and } \abs{ \mathbf{\Pi}_h \mathbf{U} }_{ \uparrow \downarrow }  \lesssim  h^{ s }  \norm{u}_{ H^{ s+2 }(\STdom)  }.  \]
\item For  $ \underline{\mathbf{U}}_h \in \ProdFullyDiscrSpace{k}{q}$ it holds that
\[
 \norm{u-\underline{u}_1}_{\STdata}
\lesssim   h^{ s+1 }  \norm{u}_{ H^{s+1}(\STdom)} + \tnorm{ ( \underline{\mathbf{U}}_h - \mathbf{\Pi}_h \mathbf{U},0) }.
\]
\item For $\mathbf{Y} = (y_1,y_2) \in [H_0^1(\STdom )]^2$ we have
\[ \norm{ \mathbf{\Pi}_h^{\ast}  \mathbf{Y} }_{ S^{\ast}_h } \lesssim  \norm{y_1}_{H^1(\STdom)} + \norm{y_2}_{H^1(\STdom)}. 
\]
\end{enumerate}
\end{lemma}
\begin{proof}
\begin{enumerate}[label=(\alph*)]
\item We only have to show the estimates for the terms $\StabTikh{ \mathbf{\Pi}_h \mathbf{U}}{\mathbf{\Pi}_h  \mathbf{U}  } $ and $\StabTrace{\mathbf{\Pi}_h  \mathbf{U} }{\mathbf{\Pi}_h \mathbf{U} }$ as the other terms have already been treated in \cite[Lemma 6 (a) \& Lemma 12]{BP24_wave}. Clearly, 
\[
\StabTikh{ \mathbf{\Pi}_h \mathbf{U} }{\mathbf{\Pi}_h \mathbf{U} } = \gamma h^{2s} \norm{ \Pi_h  u }_{L^2( \STdom)}^2 
\lesssim h^{2s}  \norm{ u }_{H^1( \STdom)}^2, 
\]
where Lemma \ref{lem:interp} (a) has been employed.
Using that $u|_{\Sigma} \in \FiniteDimSpace$, i.e. $\compl u = 0$ on $\Sigma$ and continuity of $\compl$ we obtain 
\begin{align*}
\StabTrace{\mathbf{\Pi}_h  \mathbf{U} }{\mathbf{\Pi}_h \mathbf{U} }  = 
\norm{ \compl ( \Pi_h u - u  )   }_{L^2(\Sigma)}^2 
\lesssim \norm{  \Pi_h u - u    }_{L^2(\Sigma)}^2 
\lesssim h^{2(s+1/2)} \norm{u}_{ H^{ s+1 }(\STdom)  }^2, 
\end{align*}
where the trace inequality (\ref{ieq:cont-trace-ieq}) and Lemma \ref{lem:interp} (a) have been used.
\item This follows immediately from the triangle inequality and interpolation results.
\item We refer to \cite[Lemma 6 (c)]{BP24_wave} for the proof.  
\end{enumerate}
\end{proof}

\subsection{Consistency and residual bounds}\label{ssection:consistency-res}
In this section we derive two important auxiliary results to prepare the error 
analysis given in Section \ref{section:error-analysis}.

\begin{lemma}\label{lem:consistency-bound}
Let $u$ be a solution of \eqref{eq:PDE+data_constraint} and set $ \mathbf{U} = (u, \partial_t u)$. Then for any  
$ \underline{\mathbf{W}}_h \in \ProdFullyDiscrSpace{k}{q}$ and $ \underline{\mathbf{Y}}_h \in  \ProdFullyDiscrSpace{ k_{\ast } }{ q_{\ast} }$ it holds for $s = \min{ \{ q,k \} }$ that 
\begin{align*}
(u -  \Pi_h u,\underline{w}_1)_{\STdata} -  S_h(\mathbf{\Pi}_h \mathbf{U}, \underline{\mathbf{W}}_h) -
A[ \mathbf{\Pi}_h \mathbf{U} , \underline{\mathbf{Y}}_h ] 
	\lesssim h^{s} \norm{u}_{H^{s+1}(\STdom)} \tnorm{ (\underline{\mathbf{W}}_h, \underline{\mathbf{Y}}_h) }. 
\end{align*}

\end{lemma}
\begin{proof}
The estimate for $(u -  \Pi_h u,\underline{w}_1)_{\STdata} -  S_h(\mathbf{\Pi}_h \mathbf{U}, \underline{\mathbf{W}}_h) $ follows immediately from Lemma \ref{lem:interp_in_stab} (a) and interpolation results. To estimate the remaining term, we use that as $\mathbf{U}$ solves \eqref{eq:PDE+data_constraint} we have $A[ \mathbf{\Pi}_h \mathbf{U} , \underline{\mathbf{Y}}_h ] = A[ \mathbf{\Pi}_h \mathbf{U} -\mathbf{U} , \underline{\mathbf{Y}}_h ]$. The result then follows by combining Lemma~\ref{lem:bfi_stab_est} (b) with the interpolation bounds. We refer to \cite[Lemma 7]{BP24_wave} for more details.
\end{proof}
The next lemma is an important step in controlling the PDE residual $\norm{\Box \underline{u}_1 }_{H^{-1}(\STdom)}$. Note that condition \eqref{eq:optimality_dual} holds for solutions of the variational formulation \eqref{eq:opt_compact_fully_disc}.
\begin{lemma}\label{lem:res-bound}
Let $ ( \underline{\mathbf{U}}_h, \underline{\mathbf{Z}}_h) \in (\ProdFullyDiscrSpace{k}{q}, \ProdFullyDiscrSpace{ k_{\ast } }{ q_{\ast} }) $ such that 
\begin{equation}\label{eq:optimality_dual}
A[ \underline{\mathbf{U}}_h, \mathbf{\Pi}_h^{\ast} \mathbf{W} ] = S_h^{\ast}(  \mathbf{\Pi}_h^{\ast} \mathbf{W}  , \underline{\mathbf{Z}}_h)  
\end{equation}
for any $ \mathbf{W} = (w_1,w_2) = (w,0)$ with $w \in H^1_0(\STdom)$ holds true. Then for  $ \mathbf{U} = (u, \partial_t u)$ with $u$ being the solution of \eqref{eq:PDE+data_constraint} we have: 
\begin{align*}
& \sum\limits_{n=0}^{N-1} \{ (\partial_t \underline{u}_1, \partial_t w_1)_{ Q^n } - a(\underline{u}_1,w_1)_{Q^n} \} \\ 
& \lesssim \left\vert \sum\limits_{n=1}^{N-1} ( \jump{\underline{u}_2^n}, w_{1,+}^n  )_{\Omega} \right\vert +  \norm{w}_{H^1(\STdom)} \left(  \tnorm{ ( \underline{\mathbf{U}}_h - \mathbf{\Pi}_h \mathbf{U} ,\underline{\mathbf{Z}}_h )   }  +  \abs{ \mathbf{\Pi}_h \mathbf{U} }_{ S_h}  \right). 
\end{align*}
\end{lemma}
\begin{proof}
This follows basically from integration by parts,  Lemma~\ref{lem:bfi_stab_est} (a) and  Lemma~\ref{lem:interp_in_stab} (c). We refer to see \cite[Lemma 8]{BP24_wave} for the details.
\end{proof}

\section{Error analysis}\label{section:error-analysis}
In this section we derive a convergence result for the finite element method defined in Section \ref{ssection:fully-discrete-method}. We first show convergence of the discretization error in a rather weak norm carried by the stabilization.

\begin{proposition}\label{prop:triple_norm_fully_disc_conv}
Let $u$ be a sufficiently regular solution of \eqref{eq:PDE+data_constraint} and set $\mathbf{U} := (u,\partial_t u) $. Let $ (\underline{\mathbf{U}}_h ,\underline{\mathbf{Z}}_h) \in \ProdFullyDiscrSpace{k}{q} \times \ProdFullyDiscrSpace{ k_{\ast} }{ q_{\ast} } $ 
be the solution of \eqref{eq:opt_compact_fully_disc}. Then for $s = \min{ \{ q,k \} }$ it holds that  
\[
\tnorm{ (\underline{\mathbf{U}}_h - \mathbf{\Pi}_h \mathbf{U} , \underline{\mathbf{Z}}_h ) } \lesssim  h^{s}  \norm{u}_{ H^{s+2}(\STdom)} + \norm{ \delta u }_{L^2(\STdata) }.  
\]
\end{proposition}
\begin{proof}
From \eqref{eq:opt_compact_fully_disc}-\eqref{eq:complete_bfi_def_fully_discr} and the estimates in Lemmas~\ref{lem:interp_in_stab} (a) and Lemma~\ref{lem:consistency-bound} we obtain
\begin{align*}
B[ ( \underline{\mathbf{U}}_h - \mathbf{\Pi}_h \mathbf{U}, \underline{\mathbf{Z}}_h ), ( \underline{\mathbf{W}}_h ,\underline{\mathbf{Y}}_h ) ] 
= &  (u -  \Pi_h u,\underline{w}_1)_{\STdata} - S_h(\mathbf{\Pi}_h \mathbf{U}, \underline{\mathbf{W}}_h) - A[ \mathbf{\Pi}_h \mathbf{U} , \underline{\mathbf{Y}}_h] \\
	& \!\! - \StabTimeJumps{ \mathbf{\Pi}_h \mathbf{U}}{\underline{\mathbf{W}}_h } + (\delta u,\underline{w}_1)_{\STdata}  \\
  & \lesssim  \left( h^{s}  \norm{u}_{ H^{s+2}(\STdom) } + \norm{ \delta u }_{L^2(\STdata) } \right) \tnorm{ (\underline{\mathbf{W}}_h, \underline{\mathbf{Y}}_h) }.
\end{align*}
The claim now follows from the inf-sup condition \eqref{eq:inf-sup}.
\end{proof}
The last step of the analysis consists of boosting the result from Proposition~\ref{prop:triple_norm_fully_disc_conv} by means of the stability estimates from Section \ref{section:stab} to obtain convergence in physically interesting norms. As the stability result applies to global $H^1$-functions only, we first need to account for potential discontinuities of $\underline{u}_1$ across time-slab boundaries. To this end, we introduce a lifting operator $L_{h}:\FullyDiscrSpace{k}{q} \rightarrow C^0((0,T);V_h^k)$ as follows. 
Let $\vartheta_{n}(t) := (t_{n+1} - t)/(t_{n+1} - t_n )$  for $ t \in I_n$ and set 
\begin{equation}\label{eq:def_lifting}
L_{h} \underline{w}(t) = \underline{w}(t) - \jump{ \underline{w}^n }  \vartheta_{n}(t), \quad  t \in I_n, n \geq 1, \quad
	L_{h} \underline{w}(t) = \underline{w}(t),  \quad t \in I_0.
\end{equation}
We have $\vartheta_{n}(t_n) = 1$ and $ \vartheta_{n}(t_{n+1} ) =  0$,  
which implies 
\[
\lim_{ t \downarrow t_n} L_{h} \underline{w}(t)  
= \underline{w}_{+}^{n} - \jump{ \underline{w}^n }
= \underline{w}_{-}^{n} 
= \lim_{ t \uparrow t_n} L_{h} \underline{w}(t).
\]
Hence, $ L_{h} \underline{w}$ is continuous in time. For later purposes we also record that 
due to $\abs{t_{n+1} - t_n} \sim h$ we have
\begin{equation}\label{eq:theta-and-deriv-L2-norm}
\norm{ \vartheta_{n}}_{ L^2(I_n) } \sim h^{1/2},
\qquad 
\norm{ \vartheta_{n}^{\prime} }_{ L^2(I_n) } \sim  h^{-1/2}. 
\end{equation}

Before we can establish the main convergence result we first prove an auxiliary lemma 
which allows us to control the PDE residual of the lifted error in the $H^{-1}$ norm.
\begin{lemma}\label{lem:PDE-res-bound}
Let $u$ be a sufficiently regular solution of \eqref{eq:PDE+data_constraint} and $ (\underline{\mathbf{U}}_h ,\underline{\mathbf{Z}}_h) \in \ProdFullyDiscrSpace{k}{q} \times \ProdFullyDiscrSpace{ k_{\ast} }{ q_{\ast} } $ be the solution of \eqref{eq:opt_compact_fully_disc}. 
Then for $\tilde{e}_h := u - L_{h} \underline{u}_1$ and  $s = \min{ \{ q,k \} }$ it holds that
\begin{equation}\label{eq:PDE-res-bound}
\norm{ \tilde{e}_h }_{L^2(\STdata)} + \norm{\tilde{e}_h }_{H^{-1}(\STdom)} \lesssim
h^{s}  \norm{u}_{ H^{s+2}(\STdom)} + \norm{ \delta u }_{L^2(\STdata) }.
\end{equation}
\end{lemma}
\begin{proof}
The proof is very similar to the one of \cite[Theorem 14]{BP24_wave}. Hence, we only provide 
a sketch for completeness.\\ 
For any $w \in H_0^1(\STdom)$ we have using  $\Box u = 0$, Lemma \ref{lem:res-bound}, Proposition \ref{prop:triple_norm_fully_disc_conv} and Lemma \ref{lem:interp_in_stab}(a)
\begin{align*}
	& \int\limits_{Q} ( - \partial_t \tilde{e}_h \partial_t w + \nabla \tilde{e}_h  \nabla w)
 =   \sum\limits_{n=0}^{N-1} \{ (\partial_t \underline{u}_1, \partial_t w_1)_{ Q^n } - a(\underline{u}_1,w_1)_{Q^n} \} + \xi( \underline{u}_1, w) \\
	& \lesssim ( h^{s}  \norm{u}_{ H^{s+2}(\STdom)  } + \norm{ \delta u }_{L^2(\STdata) } )  \norm{w}_{H^1(\STdom)} +  \xi( \underline{u}_1, w) + \eta( \underline{u}_2, w), 
\end{align*}
where   
\begin{align*}
\xi( \underline{u}_1, w) :=  \sum\limits_{n=1}^{N-1} \{ a( \vartheta_{n} \jump{ \underline{u}_1^n } , w  )_{Q^n}  -( \vartheta^{\prime}_{n} \jump{ \underline{u}_1^n } , \partial_t w  )_{Q^n} \}, \;  
\eta( \underline{u}_2, w) := \left\vert \sum\limits_{n=1}^{N-1} ( \jump{\underline{u}_2^n}, w_{+}^n  )_{\Omega} \right\vert. 
\end{align*}
These terms can be controlled using \eqref{eq:theta-and-deriv-L2-norm}, Lemma \ref{lem:interp} (c)-(d) and an inverse inequality in time to arrive at 
\begin{align*}
\xi( \underline{u}_1, w) + \eta( \underline{u}_2, w) \lesssim  \abs{ \underline{\mathbf{U}}_h }_{ \uparrow \downarrow } \norm{ w }_{ H^1(\STdom) } \lesssim 
( h^{s}  \norm{u}_{ H^{s+2}(\STdom)} + \norm{ \delta u }_{L^2(\STdata) } ) \norm{ w }_{ H^1(\STdom) },
\end{align*}
where the last step uses Proposition \ref{prop:triple_norm_fully_disc_conv} and Lemma \ref{lem:interp_in_stab} (a). This implies the bound \eqref{eq:PDE-res-bound} for $\norm{\tilde{e}_h }_{H^{-1}(\STdom)}$. The data fitting term can be bounded as well using \eqref{eq:theta-and-deriv-L2-norm}, Lemma \ref{lem:interp_in_stab} (b) and Proposition \ref{prop:triple_norm_fully_disc_conv}: 
\begin{align*}
\norm{ \tilde{e}_h }_{L^2(\STdata)} & \lesssim \norm{ \underline{u}_1 - L_{h} \underline{u}_1  }_{L^2(\STdata)} + \norm{  u - \underline{u}_1  }_{L^2(\STdata)} \\ 
	& \lesssim h \abs{ \underline{\mathbf{U}}_h }_{ \uparrow \downarrow }  + h^{ s+1 }  \norm{u}_{ H^{s+1}(\STdom)} + \tnorm{ ( \underline{\mathbf{U}}_h - \mathbf{\Pi}_h \mathbf{U},0) }_h \\
	& \lesssim h^{s}  \norm{u}_{ H^{s+2}(\STdom)} + \norm{ \delta u }_{L^2(\STdata) }. 
\end{align*}
\end{proof}
Now we are in a position to establish the main result. 
\begin{theorem}\label{thm:error-estimate-fully-discrete}
Let $u$ be a sufficiently regular solution of \eqref{eq:PDE+data_constraint} and $ (\underline{\mathbf{U}}_h ,\underline{\mathbf{Z}}_h) \in \ProdFullyDiscrSpace{k}{q} \times \ProdFullyDiscrSpace{ k_{\ast} }{ q_{\ast} } $ be the solution of \eqref{eq:opt_compact_fully_disc}. Moreover, let  $s = \min{ \{ q,k \} }$.  
\begin{enumerate}[label=(\alph*)]
\item Assume that no information on $u|_{\Sigma}$ is known, i.e.\ $u \in \FiniteDimSpace = L^2(\Sigma)$ and $Q = 0$. Then under the condition $\gamma > 0$ the following error estimate for $B$ as defined in \eqref{eq:B-def-hoelder} holds true 
\begin{equation}\label{eq:error-estimate-Hoelder}
\norm{ u - \underline{u}_1 }_{L^2(B)} \lesssim  h^{\alpha s} \left( \norm{u}_{ H^{s+2}(\STdom)  } + h^{-s} \norm{ \delta u }_{L^2(\STdata) }  \right).
\end{equation}
\item Assume that we are given a finite dimensional space $\FiniteDimSpace \subset L^2(\Sigma)$ such that $u \in  \FiniteDimSpace$ and that assumption \eqref{eq:assum_2}  and the geometric control condition hold true. Then we have the error estimate 
\begin{align}
& \norm{ u - L_h \underline{u}_1 }_{L^{\infty}([0,T];L^2(\Omega))} + \norm{\p_t (u - L_h \underline{u}_1)   }_{L^{2}([0,T];H^{-1}(\Omega))} \nonumber \\
 & \lesssim h^{s}  \norm{u}_{ H^{s+2}(\STdom)} + \norm{ \delta u }_{L^2(\STdata) } \label{eq:error-estimate-Lipschitz}.
\end{align}
\end{enumerate}
\end{theorem}
\begin{proof}
Let us define $\tilde{e}_h := u - L_{h} \underline{u}_1 \in H^1(\STdom)$.
\begin{enumerate}[label=(\alph*)]
\item Applying the stability estimate from Lemma~\ref{lem:stab} to $\tilde{e}_h$ gives 
\begin{align*}
\norm{ \tilde{e}_h }_{L^2(B)} & \lesssim (\norm{ \tilde{e}_h }_{L^2( \STdata  )}  + \norm{\Box  \tilde{e}_h }_{H^{-1}(Q)})^\alpha \norm{ \tilde{e}_h }_{L^2(Q)}^{1-\alpha} \\
& \lesssim ( h^{s}  \norm{u}_{ H^{s+2}(\STdom)} + \norm{ \delta u }_{L^2(\STdata) } )^{\alpha} \norm{ \tilde{e}_h }_{L^2(Q)}^{1-\alpha}, 
\end{align*}
where Lemma \ref{lem:PDE-res-bound} has been used. To estimate $\norm{ \tilde{e}_h }_{L^2(Q)}$, we write $ \tilde{e}_h = (  \underline{u}_1 - L_{\Delta t} \underline{u}_1 ) + (u - \underline{u}_1) $ and estimate the two contributions separately. The first term is easily controlled by using \eqref{eq:theta-and-deriv-L2-norm} and appealing to Proposition~\ref{prop:triple_norm_fully_disc_conv} and Lemma~\ref{lem:interp_in_stab}: 
\begin{align}
\norm{ \underline{u}_1 -  L_{h} \underline{u}_1 }_{ \STdom }
& \lesssim h^{1/2} \left( \sum\limits_{n=1}^{N-1} \norm{ \jump{ \underline{u}_1^n } }_{ \Omega}^2 \right)^{1/2} \nonumber \\ 
	& \lesssim h \abs{ \underline{\mathbf{U}}_h }_{ \uparrow \downarrow } \lesssim  h \left( h^{s} \norm{u}_{ H^{s+2}(\STdom) } +  \norm{ \delta u }_{L^2(\STdata) } \right). \label{eq:discr-jumps-bound}
\end{align}
To bound the second term we require the Tikhonov regularization, i.e. the assumption that $\gamma >0$. Using Proposition~\ref{prop:triple_norm_fully_disc_conv} and interpolation estimates gives 
\begin{align*}
& \norm{ u - \underline{u}_1 }_{L^2(\STdom)} \leq \norm{ u - \Pi_h u }_{L^2(\STdom)} 
	+ h^{-s} h^{s} \norm{ \underline{u}_1 - \Pi_h u }_{ L^2(\STdom) } \\
& \lesssim \norm{ u - \Pi_h u }_{L^2(\STdom)} + h^{-s} \tnorm{ (\underline{\mathbf{U}}_h - \mathbf{\Pi}_h \mathbf{U} , 0 ) } \lesssim \norm{u}_{ H^{s+2}(\STdom)  } + h^{-s} \norm{ \delta u }_{L^2(\STdata) }. 
\end{align*}
Combining these results yields 
\begin{align*}
\norm{ \tilde{e}_h }_{L^2(B)} &\lesssim \left( h^{s} \norm{u}_{ H^{s+2}(\STdom) } + \norm{ \delta u }_{L^2(\STdata) } \right)^{\alpha} \left( \norm{u}_{ H^{s+2}(\STdom) } + h^{-s} \norm{ \delta u }_{L^2(\STdata) } \right)^{1-\alpha}  \\
& \lesssim  h^{\alpha s} \left( \norm{u}_{ H^{s+2}(\STdom)  } + h^{-s} \norm{ \delta u }_{L^2(\STdata) } \right) .
\end{align*}
The proof is concluded by applying the triangle inequality
\[
 \norm{ u - \underline{u}_1 }_{L^2(B)} \leq  \norm{ \tilde{e}_h }_{L^2(B) } + \norm{ \underline{u}_1 - L_{h} \underline{u}_1  }_{L^2(B)}   
\]
and noting that the second term is bounded from above by \eqref{eq:discr-jumps-bound}.
\item Applying Theorem \ref{thm:lipschitz-stab-fin} to $\tilde{e}_h$ and using Lemma \ref{lem:PDE-res-bound} yields 
\[
\norm{ \tilde{e}_h }_{L^{\infty}([0,T];L^2(\Omega))} + \norm{\p_t \tilde{e}_h }_{L^{2}([0,T];H^{-1}(\Omega))}  \lesssim
 h^{s} \norm{u}_{ H^{s+2}(\STdom) } + \norm{ \delta u }_{L^2(\STdata) } + \norm{\compl \tilde{e}_h }_{L^2(\Sigma)}.
\]
It remains to treat $\norm{\compl \tilde{e}_h }_{L^2(\Sigma)}$. We split 
\[
\norm{\compl \tilde{e}_h }_{L^2(\Sigma)} \leq \norm{ \compl (\underline{u}_1 - L_h  \underline{u}_1) }_{L^2(\Sigma)} + \norm{ \compl (u - \underline{u}_1) }_{L^2(\Sigma)}
\]
and treat the two terms separately. Using continuity of $\compl$ and the trace inequality \eqref{ieq:cont-trace-ieq} we obtain 
\begin{align*}
& \norm{  \compl ( \underline{u}_1 - L_h  \underline{u}_1 )  }_{L^2(\Sigma)}^2 
 \lesssim \norm{ \underline{u}_1 - L_h  \underline{u}_1  }_{L^2(\Sigma)}^2 
\lesssim h \sum\limits_{n=1}^{N-1} \sum\limits_{K \in \mathcal{T}_h } \norm{ \jump{ \underline{u}_1^n }  }_{ L^2(\partial \Omega \cap \partial K )  }^2 \\
 & \lesssim h \sum\limits_{n=1}^{N-1} \! \sum\limits_{K \in \mathcal{T}_h } \left( \frac{1}{h} \norm{ \jump{ \underline{u}_1^n }  }_{K }^2 + \! h \norm{ \jump{ \nabla \underline{u}_1^n }  }_{K }^2  \right) \lesssim  h \abs{ \underline{\mathbf{U}}_h }_{ \uparrow \downarrow }^2 \\ 
& \lesssim h \left( h^{s} \norm{u}_{ H^{s+2}(\STdom) } +  \norm{ \delta u }_{L^2(\STdata) } \right)^2.
\end{align*}
The remaining term is controlled by the triple norm: 
\begin{align*}
\norm{ \compl (u - \underline{u}_1) }_{L^2(\Sigma)} & \leq
\norm{ \compl (u - \Pi_h u ) }_{L^2(\Sigma)} + \norm{ \compl (\Pi_h u -  \underline{u}_1) }_{L^2(\Sigma)}  \\
 & \lesssim \norm{ u - \Pi_h u  }_{L^2(\Sigma)} + \tnorm{ (\mathbf{\Pi}_h \mathbf{U} - \underline{\mathbf{U}}_h , 0 ) } \\
 & \lesssim h^{s+1/2} \norm{u}_{H^{s+1}(\STdom)} + h^{s} \norm{u}_{ H^{s+2}(\STdom) } + \norm{ \delta u }_{L^2(\STdata) }, 
\end{align*}
where we appealed again to Proposition~\ref{prop:triple_norm_fully_disc_conv}. This concludes the proof.
\end{enumerate}

\end{proof}

\section{Numerical experiments}\label{section:numexp}
In this section we present experiments to numerically explore the stability landscape of the wave equation and corroborate the theoretical results established in Section \ref{section:stab} and Theorem \ref{thm:error-estimate-fully-discrete}. We implemented the suggested method in the finite element software \texttt{Netgen/NGSolve} \cite{JS97, JS14} using in addition the space-time functionality provided by the add-on \texttt{ngsxfem} \cite{ngsxfem2021}. Reproduction material for the numerical experiments is available at \texttt{zenodo} \cite{repro-BOPZ25}.

The experiments are devided into two parts. In Section \ref{ssection:numexp-hoelder} we will investigate the situation considered by Lemma \ref{lem:stab} in which no information on the trace of $u$ on $\Sigma$ is available, i.e.\ case (a) in Theorem \ref{thm:error-estimate-fully-discrete}. In Section \ref{ssection:numexp-trace} we will then assume to be in possession of a finite dimension space on $\Sigma$ in which the trace is contained which corresponds to case (b) in the theorem. 

\subsection{Conditional H\"older stability}\label{ssection:numexp-hoelder}

We perform numerical experiments for the geometrical setup introduced in Section~\ref{ssection:stab-hoelder}. 
We consider spatial dimension $d=2$ and choose the values for the parameters as follows: 
\begin{align*}
	R = 1,  \quad r = 3/4, \quad \beta = 0.5, \quad \rho = \rho_0 +  \frac{(\rho_1 - \rho_0)}{10}, \quad  \epsilon = 5 \times 10^{-2}. 
\end{align*}
We compute up to 
\[
T = \sqrt{\frac{\rho_1 - \rho}{1-\epsilon}} \approx 0.843.
\]
Let us mention that the asymptotic convergence rates appear to be insensitive to the specific choice of the parameters as 
long as one respects the rules specified in Section~\ref{ssection:stab-hoelder} concerning the relationship 
betwen these parameters (see e.g.\ formula for $T$ above). 

\subsubsection{Clean data}\label{sssection:numexp-hoelder-clean}
\begin{figure}
\centering
\includegraphics[width=\textwidth]{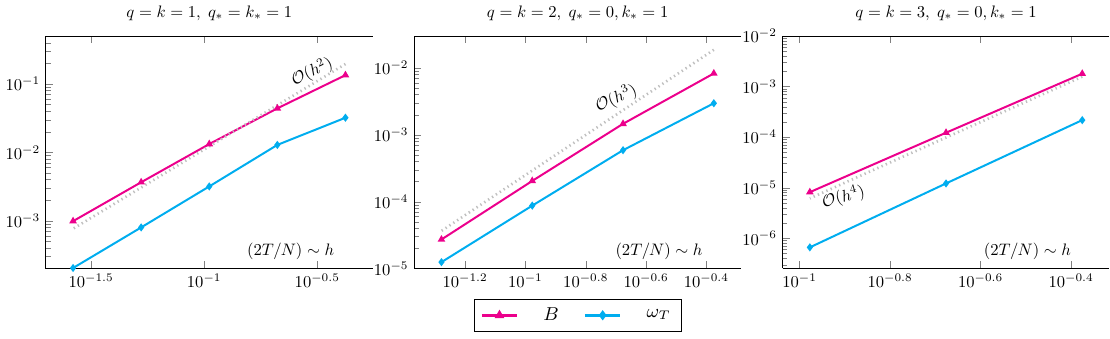}
\caption{The plots show the error $\norm{ u - \underline{u}_1 }_{L^2(D)}$ for $D \in \{  B, \omega_T  \} $ under $h$-refinement.}
\label{Cylinder-Hoelder-conv}
\end{figure}

We pick clean data $\tilde{u}_{\omega} = u|_{\STdata}$ from an exact solution $u$ of the wave equation given by
\begin{equation}\label{eq:refsol-smooth}
u = 5 \cos\left(\sqrt{2} \frac{\pi}{2} t \right) \cos\left( \frac{\pi}{2} x \right)  \cos\left( \frac{ \pi}{2}  y \right).
\end{equation}
The convergence in the set $B$ as defined in \eqref{eq:B-def-hoelder} for different polynomial degrees is investigated in Figure \ref{Cylinder-Hoelder-conv}. It turns out that the error in $B$ converges as well as the error in the data set, that is with the optimal rate $\mathcal{O}(h^{k+1})$ (for $k=q$) of the best approximation. 
Comparing this to the rate $\mathcal{O}( h^{\alpha k} )$ guaranteed by Theorem \ref{thm:error-estimate-fully-discrete} (a), we observe in particular that $\alpha$ appears to be equal to one in this case. However, as we will see later in Section \ref{sssection:numexp-hoelder-noise} this seems to be an artifact of having chosen clean data.  
\begin{figure}
\centering
\includegraphics[width=\textwidth]{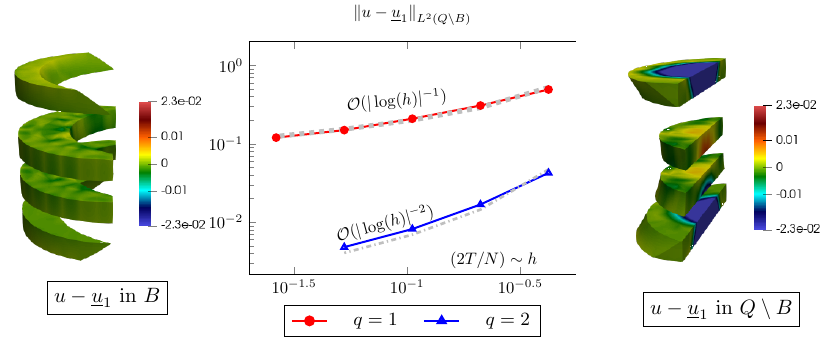}
	\caption{Plots of the difference $u - \underline{u}_1$ in $B$ and its complement $Q\setminus B$. The central plot studies the convergence in $Q\setminus B$. Here $k=q$. }
\label{Cylinder-Hoelder-conv-log}
\end{figure}

\begin{figure}
\centering
	\includegraphics[width=\textwidth]{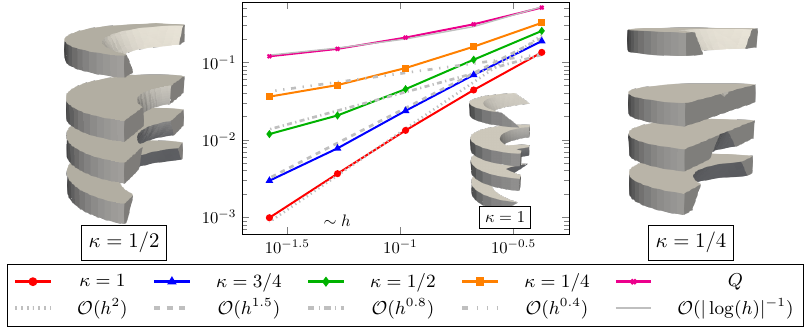}
	\caption{Convergence of the error $\norm{ u - \underline{u}_1 }_{L^2(B_{\kappa})}$ in the sets $B_{\kappa} = \mho(\kappa \rho)$ for $\kappa \in \{1,3/4,1/2,1/4 \}$ using $q=k=q_{\ast}=k_{\ast} = 1$. The figures on the left and right show the space-time domains $\mho(\kappa \rho)$ for $\kappa = 1/2$, respectively $\kappa = 1/4$.   }
\label{fig:Cylinder-Hoelder-stretch}
\end{figure}
In any case, the convergence for clean data in the set $B$ is very satisfactory. Clearly, it is also interesting to investigate the situation in the complement $\STdom \setminus B$. Note that this region is not covered by the theoretical results presented in Section \ref{section:stab}. Figure \ref{Cylinder-Hoelder-conv-log} shows that a logarithmic convergence is observed in $\STdom \setminus B$ which suggests that a corresponding logarithmic stability may hold up to the boundary $\Sigma$. Figure \ref{fig:Cylinder-Hoelder-stretch} investigates what happens in between these two extremes. To this end, let us define for $0 < \kappa \leq 1$ the regions  
\[
B_{\kappa} = \mho(\kappa \rho)
\]
with $\mho$ as defined in \eqref{eq:mho-def}. Note that $B_{1} = B$ and that the regions $B_{\kappa}$ for $\kappa < 1$ represent blown-up versions of the region $B$ which approach the boundary $\Sigma$ as $\kappa$ goes to zero. We refer to Figure \ref{fig:Cylinder-Hoelder-stretch} for sketches of $B_{\kappa}$ for $\kappa \in  \{1,1/2,1/4 \}$. The plot of the $L^2$-error displayed in this figure shows that there is a continuous decline of the convergence rate from $\mathcal{O}(h^2)$ to $\mathcal{O}(\abs{\log(h)}^{-1})$ as $\kappa$ approaches zero. 

\subsubsection{Noisy data}\label{sssection:numexp-hoelder-noise}
Let us keep the solution $u$ from \eqref{eq:refsol-smooth} and perturb the data $\tilde{u}_{\omega} = u|_{\STdata} + \delta$ with some noise $\delta$. We first test a perturbation coming from the smooth function $u \cos(2\pi x)$ projected to $\FullyDiscrSpace{k}{q} $, i.e.\ we set $\delta \sim \delta u^s$ with 
\begin{equation}\label{eq:smooth-noise}
\delta u^s := \Pi_h ( u \cos(2\pi x))
\end{equation}
with the proportionality factor chosen such that $\norm{ \delta u }_{L^2(\STdata) } \sim h^{\theta}$ for $\theta \in [1,2]$. According to Theorem \ref{thm:error-estimate-fully-discrete} a behavior of the form  $\norm{ u - \underline{u}_1 }_{L^2(B)} \lesssim h^{\alpha  }  \norm{u}_{ H^{2}(\STdom)  } +  h^{ \theta - (1-\alpha) }$ is expected for $q=k=1$.
However, what we actually observe in Figure \ref{fig:Cylinder-Hoelder-noise} is a convergence of the form $\mathcal{O}(h^{ \theta})$ which suggest that $\alpha = 1$ (and that the first factor converges as $h^{2}  \norm{u}_{ H^{2}(\STdom)  }$ as already observed before). 
\begin{figure}
\centering
\includegraphics[width=\textwidth]{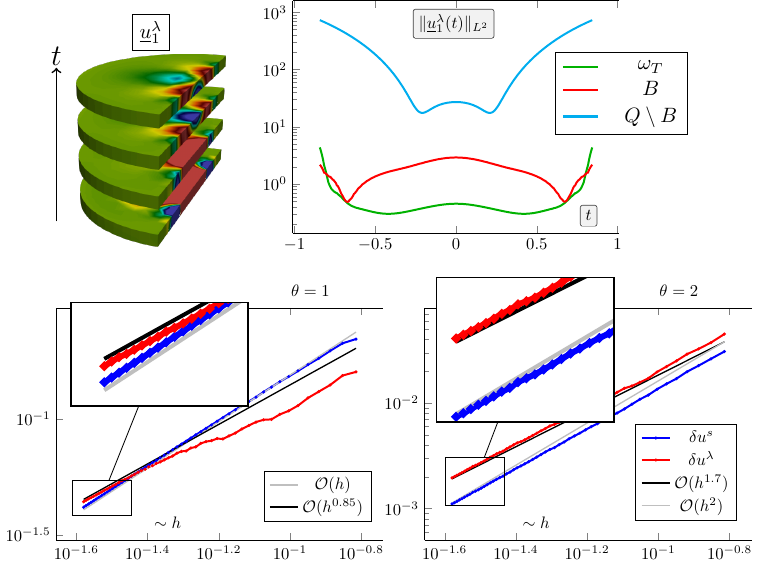}
\caption{Upper panel: Distribution of the mass of the mode $\underline{u_1}^{\lambda}$ over time. 
Lower panel: Convergence of the error $\norm{ u - \underline{u}_1 }_{L^2(B)}$ for different types of noise $\delta u \in \{ \delta u^s, \delta u^{\lambda} \}$ defined in \eqref{eq:smooth-noise}, respectively \eqref{eq:bad-mode-noise}.
The noise is scaled such that $\norm{ \delta u }_{L^2(\STdata) } \sim h^{\theta}$ for $\theta \in [1,2]$. }
\label{fig:Cylinder-Hoelder-noise}
\end{figure}
In order to make visible that the stability in $B$ is indeed reduced, and that in fact $\alpha < 1$ holds, we require a very specific kind of noise. Intuitively, if the observability in $\STdom \setminus B$ indeed fails, there should exist a solution of the wave equation which is exponentially small in the data set $\STdata$ yet very large outside. Adding this mode as noise in $\STdata$ should then trigger a large perturbation in $B$ and make the actual degree of ill-posedness visible. To find such a mode, we search for the minimal eigenvalue $\lambda $ of the problem 
\[
\begin{bmatrix}
\mathbf{S} &  \mathbf{A} \\
\mathbf{A}^{T} & \mathbf{S}^{\ast}
\end{bmatrix}
\begin{bmatrix}
   \underline{\mathbf{U}}_h^{\lambda} \\
\underline{\mathbf{Z}}_h^{\lambda}
\end{bmatrix}
=
\lambda
\begin{bmatrix}
\mathbf{M}_u &   0\\
 0  &  \mathbf{M}_z
\end{bmatrix}
\begin{bmatrix}
	\underline{\mathbf{U}}_h^{\lambda} \\
\underline{\mathbf{Z}}_h^{\lambda}
\end{bmatrix},
\]
where the matrix on the left corresponds to the bilinear form\footnote{For ease of notation we have included the data fitting term into the stabilization $\mathbf{S}$.} \eqref{eq:complete_bfi_def_fully_discr} and $\mathbf{M}_u$ and $\mathbf{M}_z$ are the mass matrices for primal and dual variable respectively. Note that for small $\lambda$ the corresponding mode will (approximately) solve the wave equation and will also be very small in $\STdata$ due to the data fidelity term. The generalized eigenvalue problem is solved iteratively using a variant of the Arnoldi method implemented in the software \texttt{ARPACK} \cite{ARPACK98}.
For  $N=64$ and $s=q=k=1$ one finds $\lambda \approx 6 \cdot 10^{-4}$. 
The corresponding mode $\underline{u_1}^{\lambda}$, extracted from $\underline{\mathbf{U}}_h^{\lambda} = ( \underline{u_1}^{\lambda},\underline{u_2}^{\lambda} )$, is visualized in Figure \ref{fig:Cylinder-Hoelder-noise}. The corresponding plot of $\norm{ \underline{u_1}^{\lambda}(t) }_{U}$ for $U \in \{ \STdata, B , \STdom \setminus B \} $ shows that this mode is strongly concentrated towards the axis of the cylinder as desired. 
Imposing this mode as noise, i.e.\ setting
\begin{equation}\label{eq:bad-mode-noise}
\delta u^{\lambda} := \Pi_h ( \underline{u_1}^{\lambda}) 
\end{equation}
we indeed observe a reduced convergence in $B$ as the red line in the lower panel of Figure \ref{fig:Cylinder-Hoelder-noise} shows. The numerical results suggest that $\alpha < 0.85$. We only give an upper bound here on the stability constant as it is possible that other types of noise could yield yet lower values. In any case, this investigation indicates that the stability in $B$ appears to be indeed of H\"older-type, contrary to the experiments with clean data that suggested that full Lipschitz stability in $B$ would hold true.

\subsection{Finite dimensional trace}\label{ssection:numexp-trace}
\begin{figure}
\centering
\includegraphics[width=\textwidth]{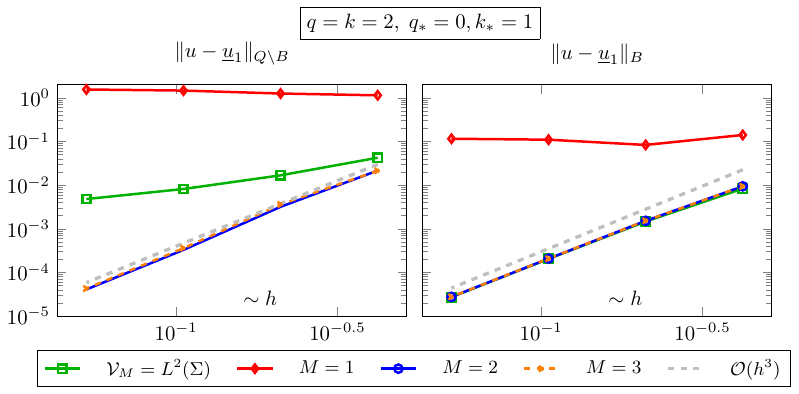}
\caption{The error $\norm{ u - \underline{u}_1 }_{L^2(U)}$ for $U \in \{ Q \setminus B, B  \} $ for different dimensions $M$ of the space  $\mathcal{V}_M := \Span \left\{ \varphi_m \mid m \in \{1, \ldots, M \} \right\}$ where the exact solution is given by $u = 5 \varphi_2$.  }
\label{fig:Cylinder-trace}
\end{figure}
Keeping the geometry from the previous section, let us now assume that we know a finite dimensional space in which the trace of the solution on $\Sigma$ is contained. For $M \in \mathbb{N}$ this space will be given by  
\begin{equation}\label{eq:V_M_Fourier}
\FiniteDimSpace_M := \Span \left\{ \varphi_m \mid m \in \{1, \ldots, M \} \right\} 
\end{equation}
for the family of solutions of the wave equation given by
\[
\varphi_m = \cos\left(\sqrt{2} \frac{m \pi}{4} t \right) \cos\left( \frac{m \pi}{4} x \right)  \cos\left( \frac{m \pi}{4}  y \right). 
\]
Let us fix the exact solution $5 \varphi_{2}$ whose trace on $\Sigma$ is obviously contained in $\FiniteDimSpace_M$ for $M \geq 2$. Notice that the stabilization term\footnote{or its modified version described in Section \ref{ssection:impl-trace-stab}} $\StabTrace{ \cdot }{ \cdot }$ now becomes active to impose that $(\underline{u}_1)|_{\Sigma}$ takes values in $\FiniteDimSpace_M$. 
Figure \ref{fig:Cylinder-trace} display the convergence rates in $B$ and in $\STdom \setminus B$ obtained for different dimensions $M$ of the space $\FiniteDimSpace_M$. If $M \geq 2$, 
i.e.\ as soon as $u|_{\Sigma} \in \FiniteDimSpace_M$, we obtain the same convergence in 
$\STdom \setminus B$ as in $B$. This confirms the theoretical result established in Theorem \ref{thm:error-estimate-fully-discrete} (b). While full Lipschitz stability is recovered when being in possession of the correct finite dimensional space, Figure \ref{fig:Cylinder-trace} also shows the devastating effect of picking the wrong space. If $M=1$ an incorrect boundary condition is imposed on $\Sigma$ which spoils convergence in $\STdom \setminus B$ as well as in $B$. 
\begin{figure}
\centering
\includegraphics[width=\textwidth]{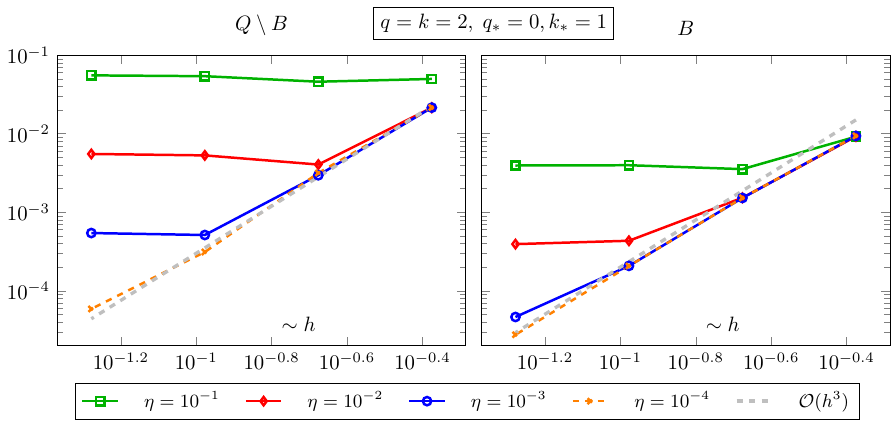}
\caption{ The error $\norm{ u_{\eta} - \underline{u}_1 }_{L^2(U)}$ for $U \in \{ Q \setminus B, B  \} $ obtained using the space $\FiniteDimSpace_2$ for the exact solution \eqref{eq:perturb-sol-eta} depending on $\eta$. }
\label{fig:Cylinder-Trace-msol3-approx}
\end{figure}

In practice, the space $\FiniteDimSpace_M$ is typically obtained from data, e.g. by model order reduction techniques, in which case it can usually not be guaranteed that the trace of the solution is contained exactly in this space. A more realistic assumption is that it can be represented only up to a certain accuracy. To emulate this situation, we consider the exact solution 
\begin{equation}\label{eq:perturb-sol-eta}
u_{\eta} = 5 \cdot \varphi_2 + \eta \varphi_3
\end{equation}
for some $\eta > 0$. We will assume to be in possession of the space $\FiniteDimSpace_M$ 
with $M =2$ in which $u_{\eta}$ can only be represented up to a certain error controlled by $\eta$. Figure \ref{fig:Cylinder-Trace-msol3-approx} demonstrates that in this situation optimal convergence is obtained everywhere up to the point where the perturbation caused by $\eta$ prevents the error from decreasing any further.

As a final experiment let us study the dependence of the error on the dimension of the space $\FiniteDimSpace_M$. Clearly one expects that choosing $\FiniteDimSpace_M$ too large should have adverse effects. For if this was not the case one could always recover Lipschitz stability everywhere by enriching with a sufficiently large space which would then somewhat be in disarray with our previous numerical results. For our investigations let us define the solution $u_M = 5 \varphi_M$. 
The smallest space in which the trace of the solution is contained is clearly 
\begin{equation}\label{eq:V-fin-dim-opt}
\FiniteDimSpace_M^{\mathrm{opt}} := \Span \left\{ \varphi_M  \right\}. 
\end{equation}
\begin{figure}
\centering
\includegraphics[width=.45\textwidth]{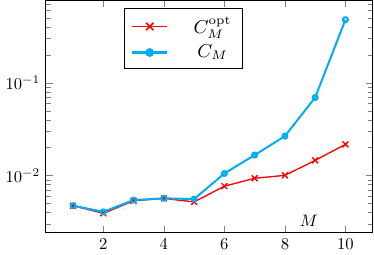}
\caption{Dependence of the quantities $C_M^{\mathrm{opt}}$ and $C_M$ 
defined in \eqref{eq:C-trace-def} on $M$.}
\label{fig:finite-dim-trace-constant} 
\end{figure}
Let $\underline{u}_1^{M,\mathrm{opt} }$ be the numerical solution obtained based on the one-dimensional space \eqref{eq:V-fin-dim-opt} while $\underline{u}_1^{M}$ denotes the solution based on the $M$-dimensional space from \eqref{eq:V_M_Fourier}. Figure \ref{fig:finite-dim-trace-constant} displays how the quantities defined by, 
\begin{equation}\label{eq:C-trace-def}
C_M^{\mathrm{opt}} :=  \frac{\norm{ u_{M} - \underline{u}_1^{M,\mathrm{opt} }  }_{L^2( \STdom )}}{  h^3 \norm{ u_{M} }_{H^{3}( \STdom)} },
\quad 
C_M :=  \frac{\norm{ u_{M} - \underline{u}_1^{M}  }_{L^2( \STdom )}}{  h^3 \norm{ u_{M} }_{H^{3}( \STdom)} }
\end{equation}
for $h = T/8$ and a fixed discretization using $q=k=2$, depend on $M$. Apparently, the quantity $C_M$ grows much faster than $C_M^{\mathrm{opt}}$ as $M$ increases. 
This means that even though we will obtain the optimal convergence 
\[
\norm{ u_{M} - \underline{u}_1^{M}  }_{L^2( \STdom)} = C_M h^3 \norm{ u_{M} }_{H^{3}( \STdom)} 
\]
based on the space $\FiniteDimSpace_M$ as $h$ goes to zero, the obtained accuracy will be very poor for $M > 9$ due to the exponential growth of the constant $C_M$. This result is plausible as it shows that the ill-posedness of the problem can only be contained by having accurate additional information on the solution available.





\bibliographystyle{elsarticle-num} 
\bibliography{master}

\end{document}